\documentclass[11pt]{article}

\usepackage{amsfonts,amsmath,amsthm,amssymb}
\usepackage{bm}
\usepackage[colorlinks, citecolor=blue, urlcolor=blue]{hyperref}
\usepackage[numbers]{natbib}
\usepackage{geometry}
\usepackage[utf8]{inputenc}
\usepackage[T1]{fontenc}    
\usepackage{mathtools}

\usepackage{enumitem}

\newtheorem{theorem}{Theorem}[section]
\newtheorem{corollary}[theorem]{Corollary}
\newtheorem{lemma}[theorem]{Lemma}
\newtheorem{proposition}[theorem]{Proposition}

\newtheorem{example}[theorem]{Example}
\theoremstyle{definition}
\newtheorem{definition}[theorem]{Definition}

\newtheorem{remark}[theorem]{\textbf{Remark}}
\numberwithin{equation}{section}

\begin{document}

\title{Markovian projections for functionals of It\^o semimartingales with jumps\footnote{Acknowledgement: the authors would like to thank Steven Shreve for useful conversations.}}
\author{
	Martin Larsson\footnote{Department of Mathematical Sciences, Carnegie Mellon University, \texttt{larsson@cmu.edu}} \and
	Shukun Long\footnote{Department of Mathematical Sciences, Carnegie Mellon University, \texttt{shukunl@alumni.cmu.edu}}
}
\maketitle

\begin{abstract}
Given an It\^o semimartingale $X$, its Markovian projection is an It\^o semimartingale $\widehat{X}$, with Markovian differential characteristics, that matches the one-dimensional marginal laws of $X$. One may even require certain functionals of the two processes to have the same fixed-time marginals, at the cost of enhancing the differential characteristics of $\widehat{X}$ but still in a Markovian sense. In the continuous case, the definitive result on existence of Markovian projections was obtained by Brunick and Shreve~\cite{MR3098443}. In this paper, we extend their result to the fully general setting of It\^o semimartingales with jumps.
\end{abstract}

\section{Introduction}\label{sec:1}
The Markovian projection of an It\^o process is another It\^o process with Markovian-type dynamics that mimics the one-dimensional marginal laws of the original process. Let us say we are given a $d$-dimensional continuous It\^o process $X$ with dynamics
\begin{equation}\label{eq:SDE1}
	X_t = X_0 + \int_0^t b_s \,ds + \int_0^t \sigma_s \,dW_s,
\end{equation}
where $b$ and $\sigma$ are predictable processes taking values in $\mathbb{R}^d$ and $\mathbb{R}^{d \times k}$ respectively, and $W$ is a $k$-dimensional Brownian motion. The processes $b$ and $\sigma$ may be solutions to some exogenous SDEs with path dependent features, making the dynamics of $X$ much more complicated. Our goal is to find a simpler process $\widehat{X}$, possibly defined on a different probability space, that solves a Markovian SDE
\begin{equation}\label{eq:SDE2}
	\widehat{X}_t = \widehat{X}_0 + \int_0^t \widehat{b}(s, \widehat{X}_s) \,ds + \int_0^t \widehat{\sigma}(s, \widehat{X}_s) \,d\widehat{W}_s,
\end{equation}
such that for every $t \geq 0$, the law of $\widehat{X}_t$ agrees with the law of $X_t$. Here $\widehat{b}$ and $\widehat{\sigma}$ are deterministic functions taking values in $\mathbb{R}^d$ and $\mathbb{R}^{d \times d}$ respectively, and $\widehat{W}$ is a $d$-dimensional Brownian motion. If we manage to do so, the process $\widehat{X}$ is called a \emph{Markovian projection} of $X$. We emphasize that although the terminology ``Markovian projection'' has the word ``Markov'' in it, we only require $\widehat{X}$ to solve a Markovian SDE. Without further regularity assumptions on the coefficients $\widehat{b}$ and $\widehat{\sigma}$, we know $\widehat{X}$ may not necessarily be a true Markov process. Also, some authors use alternative terminologies like ``mimicking process'' when referring to the process $\widehat{X}$, and ``mimicking theorem'' when referring to the results that construct such an $\widehat{X}$. The goal of this paper is to establish existence of Markovian projections in the fully general setting of It\^o semimartingales, including processes with jumps.

The idea of Markovian projections was first introduced in the seminal work of Krylov~\cite{MR0808203} and Gy\"{o}ngy~\cite{MR0833267}. In \cite{MR0833267}, Gy\"{o}ngy constructed Markovian projections for continuous It\^o processes $X$ as formulated in \eqref{eq:SDE1} with $X_0 = 0$. Gy\"{o}ngy's results hold under a boundedness assumption on the coefficients $b$ and $\sigma$, and a uniform ellipticity condition on the matrix-valued process $\sigma \sigma^\top$. The mimicking process $\widehat{X}$ is constructed as a weak solution to the SDE \eqref{eq:SDE2}, where the coefficients $\widehat{b}$ and $\widehat{\sigma}$ have an explicit expression:
\begin{equation}\label{eq:cond_exp_intro}
	\begin{split}
		\widehat{b}(t, x)
		&= \mathbb{E}[b_t \,|\, X_t = x],\\
		\widehat{\sigma}(t, x) \widehat{\sigma}(t, x)^\top
		&= \mathbb{E}[\sigma_t \sigma_t^\top \,|\, X_t = x].
	\end{split}
\end{equation}
Strictly speaking, the conditional expectations above should be understood as certain Radon--Nikodym derivatives, but \eqref{eq:cond_exp_intro} provides an intuitive interpretation of the functions $\widehat{b}$ and $\widehat{\sigma}$. For the precise definition, the readers can refer to \cite{MR0833267}, Section~4. Gy\"{o}ngy's work on Markovian projections was inspired by Krylov~\cite{MR0808203}, where a different type of mimicking problem was studied. In \cite{MR0808203}, one of the objects of interest is called the \emph{Green measure}, which characterizes the average length of time that an It\^o process stays in a Borel set. Krylov constructed a simpler It\^o diffusion that has the same Green measure as a more general It\^o diffusion. The mimicking process of Krylov solves a Markovian SDE, with time-homogeneous coefficients. Following similar proof techniques, Gy\"{o}ngy showed that the one-dimensional marginal laws can be mimicked as well using a Markovian SDE, while the coefficients have to be time-inhomogeneous in general as in \eqref{eq:SDE2}.

Gy\"{o}ngy's theorem on Markovian projections can be extended in multiple directions. Firstly, the boundedness and non-degeneracy conditions on the coefficients $b$ and $\sigma$ are quite restrictive. It is natural to ask for weaker assumptions. On the other hand, apart from mimicking the one-dimensional marginal laws of the process $X$ itself, one may also be interested in mimicking the joint law of $X$ and some functional of $X$ at each fixed time. Both aspects were addressed in Brunick and Shreve~\cite{MR3098443}. In their work, they relaxed the assumptions on $b$ and $\sigma$ to an integrability condition:
\begin{equation}\label{eq:int_cond_cts}
	\mathbb{E}\biggl[\int_0^t (|b_s| + |\sigma_s \sigma_s^\top|) \,ds\biggr] < \infty,\quad
	\forall\, t > 0.
\end{equation}
They also proved a mimicking theorem for a class of functionals, called \emph{updating functions}, of It\^o processes. To avoid technical details here, let us consider $d=1$ and a special case of updating functions --- ``maximum-to-date''. Then, using the main theorem in \cite{MR3098443}, one can construct a Markovian projection for the pair $(X, M)$, where $M \coloneqq \max_{s \leq \cdot} X_s$. The mimicking process $\widehat{X}$, augmented by its running maximum $\widehat{M} \coloneqq \max_{s \leq \cdot} \widehat{X}_s$, follows the Markovian-type dynamics
\begin{equation*}
	\widehat{X}_t = \widehat{X}_0 + \int_0^t \widehat{b}(s, \widehat{X}_s, \widehat{M}_s) \,ds + \int_0^t \widehat{\sigma}(s, \widehat{X}_s, \widehat{M}_s) \,d\widehat{W}_s,
\end{equation*}
where the deterministic functions $\widehat{b}$ and $\widehat{\sigma}$ are given by
\begin{equation*}
	\begin{split}
		\widehat{b}(t, x, y)
		&= \mathbb{E}[b_t \,|\, X_t = x,\, M_t = y],\\
		\widehat{\sigma}^2(t, x, y)
		&= \mathbb{E}[\sigma_t^2 \,|\, X_t = x,\, M_t = y].
	\end{split}
\end{equation*}
The proof techniques in \cite{MR3098443} are purely probabilistic and completely different from those in \cite{MR0833267}, where ideas from PDE come into play.

Another natural extension of Gy\"{o}ngy's results is to consider It\^o processes with jumps. Now let us say $X$ is a $d$-dimensional c\`adl\`ag It\^o semimartingale with the canonical representation
\begin{equation*}
	\begin{split}
		X_t = X_0 &+ \int_0^t b_s \,ds + \int_0^t \sigma_s \,dW_s\\ &+ \int_0^t \int_{\{|\xi| \leq 1\}} \xi \,(\mu^X(ds, d\xi) - \kappa_s(d\xi)ds) + \int_0^t \int_{\{|\xi| > 1\}} \xi \,\mu^X(ds, d\xi),
	\end{split}
\end{equation*}
where $\mu^X$ is an integer-valued random measure on $\mathbb{R}_+ \times \mathbb{R}^d$ that charges $1$ at each point of the form ``$(\text{jump time of } X,\, \text{jump size of } X)$'', and $\kappa$ is a predictable transition kernel from $\Omega \times \mathbb{R}_+$ to $\mathbb{R}^d$. The triplet $(b, c, \kappa)$, where $c \coloneqq \sigma \sigma^\top$, is called the \emph{differential characteristics} of $X$. The goal is to construct an It\^o process $\widehat{X}$ that has Markovian-type differential characteristics:
\begin{equation*}
	\begin{split}
		\widehat{X}_t = \widehat{X}_0 &+ \int_0^t \widehat{b}(s, \widehat{X}_{s-}) \,ds + \int_0^t \widehat{\sigma}(s, \widehat{X}_{s-}) \,d\widehat{W}_s\\ &+ \int_0^t \int_{\{|\xi| \leq 1\}} \xi \,(\mu^{\widehat{X}}(ds, d\xi) - \widehat{\kappa}(s, \widehat{X}_{s-}, d\xi)ds) + \int_0^t \int_{\{|\xi| > 1\}} \xi \,\mu^{\widehat{X}}(ds, d\xi),
	\end{split}
\end{equation*}
such that for every $t \geq 0$, the law of $\widehat{X}_t$ agrees with the law of $X_t$. The functions $\widehat{b}$ and $\widehat{\sigma}$ are given by \eqref{eq:cond_exp_intro}, and analogously we expect the deterministic transition kernel $\widehat{\kappa}$ to have the following explicit expression:
\begin{equation*}
	\widehat{\kappa}(t, x, d\xi)
	= \mathbb{E}[\kappa_t(d\xi) \,|\, X_t = x].
\end{equation*}

Bentata and Cont~\cite{bentata2012mimicking} studied the above problem that involves jumps. As is the case for Gy\"{o}ngy~\cite{MR0833267}, Bentata and Cont's theorem holds under the same boundedness and non-degeneracy conditions on $b$ and $\sigma$, together with a boundedness and decay condition on the third differential characteristic $\kappa$. Moreover, they also imposed some continuity assumptions on $\widehat{b}$, $\widehat{\sigma}$ and $\widehat{\kappa}$, which are not always easy to check in practice. However, it is worth mentioning that although their assumptions are relatively strong, they also showed the uniqueness in law and the Markov property of the mimicking process, which are not guaranteed in \cite{MR0833267} or \cite{MR3098443}.

In our previous work \cite{MR4814246}, we independently developed Markovian projections for c\`adl\`ag It\^o semimartingales. One of the main tools in \cite{MR4814246} is the superposition principle for non-local generators developed by R\"{o}ckner, Xie and Zhang~\cite{MR4168386}. The idea of using a superposition principle to prove a mimicking theorem seems to have been first used in Lacker, Shkolnikov and Zhang~\cite{MR4612111}. The main results in \cite{MR4814246} hold under relatively mild assumptions: an integrability condition similar to \eqref{eq:int_cond_cts} as in \cite{MR3098443}:
\begin{equation}\label{eq:int_cond_jump}
	\mathbb{E}\biggl[\int_0^t \biggl(|b_s| + |\sigma_s \sigma_s^\top| + \int_{\mathbb{R}^d} 1 \land |\xi|^2 \,\kappa_s(d\xi)\biggr) \,ds\biggr] < \infty,\quad
	\forall\, t > 0,
\end{equation}
and a growth condition on $(\widehat{b}, \widehat{\sigma}, \widehat{\kappa})$ (see \cite{MR4814246}, Equation~(3.4)). Although this growth condition is not strictly weaker than the assumption in \cite{bentata2012mimicking}, it is generally easier to verify than a continuity condition. Also, as is in \cite{MR0833267} and \cite{MR3098443}, properties beyond existence of the mimicking process are not guaranteed in general. An important limitation of the method in \cite{MR4814246} is that the superposition principle is not available in path-dependent situations involving updating functions.

In this paper, we construct Markovian projections in the fully general setting of It\^o semimartingales with jumps. The only assumption of our new results is \eqref{eq:int_cond_jump}; the growth condition on $(\widehat{b}, \widehat{\sigma}, \widehat{\kappa})$ is no longer needed. Thus, our assumption is much weaker than those in \cite{bentata2012mimicking} and \cite{MR4814246}. Moreover, our mimicking results work for a class of functionals of It\^o processes, not just the processes themselves, which again strictly generalizes \cite{bentata2012mimicking} and \cite{MR4814246}. The proof techniques are completely different than those in \cite{MR4814246}. We build on the pioneering ideas of Brunick and Shreve~\cite{MR3098443}, utilizing in particular the concepts of updating functions and concatenated probability measures. Nonetheless, the extension to the jump case is nontrivial and requires, for instance, a carefully designed canonical space for the third characteristic of an It\^o semimartingale, the compensator of its jump measure. This involves a delicate analysis of certain measure-valued processes, which does not arise in the continuous case in \cite{MR3098443}.

This paper is organized as follows. In Section~\ref{sec:2} we state our main results. In Section~\ref{sec:3} we build the canonical space and gather all the required preliminaries. In Section~\ref{sec:4} we prove our main results. Throughout this paper, we use the following notation and convention:
\begin{itemize}[topsep=0.3em, noitemsep]
	\item $\mathbb{R}_+ = [0, \infty)$.
	
	\item $\mathbb{N}$ ($\mathbb{N}^*$) is the set of natural numbers including (excluding) $0$.
	
	\item $\mathbb{S}_+^d$ is the set of symmetric positive semi-definite $d \times d$ real matrices.
	
	\item $\mu(f) = \int f \,d\mu$, for $\mu$ a measure and $f$ a measurable function on some space such that the integral is well-defined.
	
	\item All semimartingales have c\`adl\`ag sample paths.
\end{itemize}

\section{Main Results}\label{sec:2}
In this section we present our main result, Theorem~\ref{thm:mp}. The statement of the main theorem involves a concept called the \emph{updating function}. This will be crucial when we want to mimic the one-dimensional marginal laws of functionals of It\^o processes. The proof of Theorem~\ref{thm:mp} is postponed to Section~\ref{sec:4}.

\subsection{Updating function}\label{sec:upd_fn}
Let $\mathcal{E}$ be a Polish space. Let $C^{\mathcal{E}}$ be the space of continuous functions from $\mathbb{R}_+$ to $\mathcal{E}$, endowed with the topology of uniform convergence on compact intervals. Let $D^{\mathcal{E}}$ be the space of c\`adl\`ag functions from $\mathbb{R}_+$ to $\mathcal{E}$, endowed with the Skorokhod's $J_1$ topology. If $\mathcal{E}$ is a subset of a vector space with $0 \in \mathcal{E}$, we let $C^{\mathcal{E}}_0$ (resp.\ $D^{\mathcal{E}}_0$) denote the closed subset of $C^{\mathcal{E}}$ (resp.\ $D^{\mathcal{E}}$) consisting of elements with initial value $0$. In particular, when $\mathcal{E} = \mathbb{R}^d$, we write $C^d$, $C^d_0$, $D^d$ and $D^d_0$ for short, rather than $C^{\mathbb{R}^d}$, $C^{\mathbb{R}^d}_0$, $D^{\mathbb{R}^d}$ and $D^{\mathbb{R}^d}_0$. Note that all the spaces defined here are Polish spaces.

We define three types of elementary operators on the space $D^{\mathcal{E}}$. The \emph{shift operator} $\Theta: D^{\mathcal{E}} \times \mathbb{R_+} \to D^{\mathcal{E}}$ is defined via
\begin{equation*}
	\Theta(x, t) \coloneqq x(t + \cdot),\quad
	x \in D^{\mathcal{E}},\, t \geq 0.
\end{equation*}
The \emph{stopping operator} $\nabla: D^{\mathcal{E}} \times \mathbb{R_+} \to D^{\mathcal{E}}$ is defined via
\begin{equation*}
	\nabla(x, t) \coloneqq x(t \land \cdot),\quad
	x \in D^{\mathcal{E}},\, t \geq 0.
\end{equation*}
We alternatively write $x^t = x(t \land \cdot)$. If $\mathcal{E}$ is a vector space, the \emph{difference operator} $\Delta: D^{\mathcal{E}} \times \mathbb{R_+} \to D^{\mathcal{E}}_0$ is defined via
\begin{equation*}
	\Delta(x, t) \coloneqq x(t + \cdot) - x(t),\quad
	x \in D^{\mathcal{E}},\, t \geq 0.
\end{equation*}
Note that if we restrict the operators $\Theta$, $\nabla$ and $\Delta$ to $C^{\mathcal{E}} \times \mathbb{R}_+$, then their ranges are all included in $C^{\mathcal{E}}$.

\begin{definition}[cf.\ \cite{MR3098443}, Definition~3.1]
	We say that $\Phi: \mathcal{E} \times D^d_0 \to D^{\mathcal{E}}$ is an updating function, if it satisfies
	\begin{enumerate}[label=(\roman*), topsep=0.3em, noitemsep]
		\item initial condition:
		\begin{equation*}
			\Phi(e, x)(0) = e,\quad
			\forall\, e \in \mathcal{E},\, x \in D^d_0,
		\end{equation*}
		
		\item nonanticipativity:
		\begin{equation*}
			\nabla(\Phi(e, x), t) = \nabla(\Phi(e, \nabla(x, t)), t),\quad
			\forall\, t \geq 0,\, e \in \mathcal{E},\, x \in D^d_0,
		\end{equation*}
		
		\item ``Markov property'':
		\begin{equation*}
			\Theta(\Phi(e, x), t) = \Phi(\Phi(e, x)(t), \Delta(x, t)),\quad
			\forall\, t \geq 0,\, e \in \mathcal{E},\, x \in D^d_0.
		\end{equation*}
	\end{enumerate}
\end{definition}

The updating function $\Phi$ takes an initial value in $\mathcal{E}$ and a path in $D^d_0$, then generates a path in $D^{\mathcal{E}}$. Since $\Phi$ is a map between two Polish spaces, one can also talk about its continuity. In particular, in our main results, we will require the updating functions to be continuous. Below are some examples of continuous updating functions, most of which are presented in Brunick and Shreve~\cite{MR3098443}. However, since we are extending from the ``$C$-space'' to the ``$D$-space'', it is worth discussing these examples here, especially their continuity.

\begin{example}[Process itself]\label{eg:1}
	Let $\mathcal{E} = \mathbb{R}^d$, and define $\Phi: \mathbb{R}^d \times D^d_0 \to D^d$ via
	\begin{equation*}
		\Phi(e, x) \coloneqq e + x,\quad
		e \in \mathbb{R}^d,\, x \in D^d_0.
	\end{equation*}
	If $X$ is an $\mathbb{R}^d$-valued c\`adl\`ag process, then we trivially have $\Phi(X_0, X - X_0) = X$, which recovers the process itself. Clearly, $\Phi$ is a continuous updating function.
\end{example}

\begin{example}[Integral-to-date]\label{eg:2}
	Let $\mathcal{E} = \mathbb{R}^2$, $d = 1$, and define $\Phi: \mathbb{R}^2 \times D^1_0 \to D^2$ via
	\begin{equation*}
		\Phi(e, x)
		\coloneqq \biggl(e_1 + x, e_2 + \int_0^\cdot (e_1 + x(s)) \,ds\biggr),\quad
		e = (e_1, e_2) \in \mathbb{R}^2,\, x \in D^1_0.
	\end{equation*}
	If $X$ is a real-valued c\`adl\`ag process and $A_0$ is a real-valued random variable, then we have
	\begin{equation*}
		\Phi((X_0, A_0), X - X_0)
		= (X, A),\quad
		\text{where } A_t = A_0 + \int_0^t X_s \,ds.
	\end{equation*}
	It is easy to check that $\Phi$ is an updating function. To see $\Phi$ is continuous, we only need to verify its second component $\Phi_2$. We notice that $\Phi_2$ takes values in $C^1$ (not just $D^1$), so we can prove continuity using the topology of the ``$C$-space''. Take $e^n \to e$ in $\mathbb{R}^2$ and $x^n \to x$ in $D^1_0$. It suffices to show
	\begin{equation*}
		\begin{split}
			&\max_{t \leq T}\, \biggl|e^n_2 + \int_0^t (e^n_1 + x^n(s)) \,ds - e_2 - \int_0^t (e_1 + x(s)) \,ds\biggr|\\
			&\quad\quad \leq |e^n_2 - e_2| + T|e^n_1 - e_1| + \int_0^T |x^n(s) - x(s)| \,ds \to 0,\quad
			\forall\, T > 0.
		\end{split}
	\end{equation*}
	Since $(x^n)$ and $x$ are uniformly bounded on $[0, T]$, and $x^n(s) \to x(s)$ for all but countably many values of $s$, the dominated convergence theorem finishes the proof. 
\end{example}

\begin{example}[Supremum-to-date]\label{eg:3}
	Let $\mathcal{E} = \{(e_1, e_2) \in \mathbb{R}^2: e_1 \leq e_2\}$, $d = 1$, and define $\Phi: \mathcal{E} \times D^1_0 \to D^{\mathcal{E}} \subset D^2$ via
	\begin{equation*}
		\Phi(e, x)
		\coloneqq \biggl(e_1 + x, e_2 \lor \sup_{s \leq \cdot} (e_1 + x(s))\biggr),\quad
		e = (e_1, e_2) \in \mathcal{E},\, x \in D^1_0.
	\end{equation*}
	If $X$ is a real-valued c\`adl\`ag process and $M_0$ is a real-valued random variable satisfying $M_0 \geq X_0$ a.s., then we have
	\begin{equation*}
		\Phi((X_0, M_0), X - X_0)
		= (X, M),\quad
		\text{where } M_t = M_0 \lor \sup_{s \leq t} X_s.
	\end{equation*}
	It is easy to check that $\Phi$ is an updating function. To see $\Phi$ is continuous, we only need to verify its second component $\Phi_2$. Take $e^n \to e$ in $\mathcal{E}$ and $x^n \to x$ in $D^1_0$. We know (see e.g.\ \cite{MR1700749}, Theorem~16.1) there exists a sequence $(\lambda^n)$ of continuous increasing functions from $\mathbb{R}_+$ onto $\mathbb{R}_+$ such that $\lambda^n \to \mathrm{id}$ uniformly on $\mathbb{R}_+$ and $x^n \circ \lambda^n \to x$ uniformly on compact intervals. To prove $\Phi_2(e^n, x^n) \to \Phi_2(e, x)$ in $D^1$, it suffices to show
	\begin{equation*}
		\sup_{t \leq T} \,\biggl|e^n_2 \lor \sup_{s \leq \lambda^n(t)} (e^n_1 + x^n(s)) - e_2 \lor \sup_{s \leq t} (e_1 + x(s))\biggr| \to 0,\quad \forall\, T > 0.
	\end{equation*}
	However, one can rewrite the left-hand side and bound it by
	\begin{equation*}
		\begin{split}
			&\sup_{t \leq T} \,\biggl| e^n_2 \lor \sup_{s \leq t} (e^n_1 + x^n(\lambda^n(s))) - e_2 \lor \sup_{s \leq t} (e_1 + x(s))\biggr|\\
			&\quad\quad \leq |e^n_2 - e_2| \lor \biggl(|e^n_1 - e_1| + \sup_{s \leq T} |x^n(\lambda^n(s)) - x(s)|\biggr),
		\end{split}
	\end{equation*}
	which clearly goes to $0$ by assumption.
\end{example}

\begin{example}[Maximal jump-to-date]\label{eg:4}
	Let $\mathcal{E} = \mathbb{R} \times \mathbb{R}_+$, $d=1$, and define $\Phi: \mathcal{E} \times D^1_0 \to D^{\mathcal{E}} \subset D^2$ via
	\begin{equation*}
		\Phi(e, x)
		\coloneqq \biggl(e_1 + x, e_2 \lor \max_{s \leq \cdot} (x(s) - x(s-))\biggr),\quad
		e = (e_1, e_2) \in \mathcal{E},\, x \in D^1_0.
	\end{equation*}
	If $X$ is a real-valued c\`adl\`ag process and $J_0$ is a nonnegative random variable, then we have
	\begin{equation*}
		\Phi((X_0, J_0), X - X_0)
		= (X, J),\quad
		\text{where } J_t = J_0 \lor \max_{s \leq t} \Delta X_s.
	\end{equation*}
	It is easy to check that $\Phi$ is an updating function. The continuity of $\Phi$ can be proved in almost the same way as in Example~\ref{eg:3}, once we notice the following simple fact:
	\begin{equation*}
		\max_{s \leq t} |y(s) - y(s-)|
		\leq 2 \sup_{s \leq t} |y(s)|,\quad
		\forall\, y \in D^1,\, t \geq 0.
	\end{equation*}
\end{example}

\subsection{Semimartingale characteristics}
In this subsection we briefly review the concept of semimartingale characteristics. For a detailed discussion, the readers can refer to \cite{MR1943877}, Chapter~II.2. Recall that a semimartingale $X$ is a c\`adl\`ag process which admits a decomposition $X = B + M$, where $B$ is a finite variation process and $M$ is a local martingale. Such a decomposition is not unique. A special semimartingale $X$ is a semimartingale which admits a decomposition $X = B + M$, where $B$ is a predictable finite variation process and $M$ is a local martingale. In this case, such a decomposition is unique, and is called the \emph{canonical decomposition} of $X$. In particular, by \cite{MR1943877}, Lemma~I.4.24, a semimartingale with bounded jumps is a special semimartingale.

\begin{definition}
	We say $h: \mathbb{R}^d \to \mathbb{R}^d$ is a \emph{truncation function} if $h$ is measurable, bounded and $h(x) = x$ in a neighborhood of $0$.
\end{definition}

\begin{definition}
	Let $(\Omega, \mathcal{F}, (\mathcal{F}_t)_{t \geq 0}, \mathbb{P})$ be a filtered probability space. Let $\mathcal{P}$ be the predictable $\sigma$-algebra on $\Omega \times \mathbb{R}_+$, and $\mu: \Omega \times \mathcal{B}(\mathbb{R}_+ \times \mathbb{R}^d) \to [0, \infty]$ be a random measure. We say $\mu$ is a \emph{predictable random measure}, if the process
	\begin{equation*}
		\Omega \times \mathbb{R}_+ \ni (\omega, t)
		\mapsto \int_{[0, t] \times \mathbb{R}^d} W(\omega, s, x) \,\mu(\omega, ds, dx)
	\end{equation*}
	is predictable for all nonnegative functions $W$ on $\Omega \times \mathbb{R}_+ \times \mathbb{R}^d$ which are measurable with respect to $\mathcal{P} \otimes \mathcal{B}(\mathbb{R}^d)$.
\end{definition}

\begin{definition}\label{def:char}
	Let $X = (X^i)_{1 \leq i \leq d}$ be an $\mathbb{R}^d$-valued semimartingale. The \emph{characteristics} of $X$ associated with a truncation function $h$ is the triplet $(B, C, \nu)$ consisting in:
	\begin{enumerate}[label=(\roman*), topsep=0.3em, noitemsep]
		\item $B = (B^i)_{1 \leq i \leq d}$ is an $\mathbb{R}^d$-valued predictable finite variation process, which is the predictable finite variation part of the special semimartingale
		\begin{equation*}
			X(h)_t \coloneqq X_t - \sum_{s \leq t} (\Delta X_s - h(\Delta X_s)),
		\end{equation*}
		
		\item $C = (C^{ij})_{1 \leq i, j \leq d}$ is an $\mathbb{R}^{d^2}$-valued continuous finite variation process, such that
		\begin{equation*}
			C^{ij}
			= \langle X^{i, c}, X^{j, c} \rangle,\quad
			1 \leq i, j \leq d,
		\end{equation*}
		where $X^c = (X^{i, c})_{1 \leq i \leq d}$ is the continuous local martingale part of $X$,
		
		\item $\nu$ is a predictable random measure on $\mathbb{R}_+ \times \mathbb{R}^d$, which is the compensator of the random measure $\mu^X$ associated with the jumps of $X$, namely
		\begin{equation*}
			\mu^X(dt, d\xi) \coloneqq \sum_{s > 0} \bm{1}_{\{\Delta X_s \neq 0\}} \delta_{(s, \Delta X_s)}(dt, d\xi).
		\end{equation*}
	\end{enumerate}
\end{definition}

Note that $C$ and $\nu$ do not depend on the choice of the truncation function $h$, while $B = B(h)$ does. For two truncation functions $h$ and $\widetilde{h}$, the relationship between their corresponding $B$ is given by \cite{MR1943877}, Proposition~II.2.24:
\begin{equation}\label{eq:B(h)}
	B(h)_t - B(\widetilde{h})_t
	= \int_{[0, t] \times \mathbb{R}^d} (h(\xi) - \widetilde{h}(\xi)) \,\nu(ds, d\xi).
\end{equation}

\begin{definition}
	Let $(X, \mathcal{A})$ be a measurable space, and $\kappa: X \times \mathcal{B}(\mathbb{R}^d) \to [0, \infty]$ be a transition kernel from $X$ to $\mathbb{R}^d$. We say $\kappa$ is a \emph{L\'evy transition kernel}, if $\kappa(x, dy)$ is a L\'evy measure on $\mathbb{R}^d$ for each $x \in X$, i.e.\
	\begin{equation*}
		\kappa(x, \{0\}) = 0,\quad
		\int_{\mathbb{R}^d} 1 \land |y|^2 \,\kappa(x, dy) < \infty.
	\end{equation*}
\end{definition}

\begin{definition}\label{def:diff_charac}
	Let $X$ be an $\mathbb{R}^d$-valued semimartingale with characteristics triplet $(B, C, \nu)$ associated with a truncation function $h$. We say $X$ is an \emph{It\^o semimartingale}, if there exist an $\mathbb{R}^d$-valued predictable process $b$, an $\mathbb{S}^d_+$-valued predictable process $c$, and a predictable L\'evy transition kernel $\kappa$ from $\Omega \times \mathbb{R}_+$ to $\mathbb{R}^d$, such that
	\begin{equation*}
		B_t = \int_0^t b_s \,ds,\quad
		C_t = \int_0^t c_s \,ds,\quad
		\nu([0, t] \times A) = \int_0^t \kappa_s(A) \,ds,\quad
		t \geq 0,\, A \in \mathcal{B}(\mathbb{R}^d).
	\end{equation*}
	We call the triplet $(b, c, \kappa)$ the \emph{differential characteristics} of $X$ associated with $h$.
\end{definition}

Briefly speaking, an It\^o semimartingale is a semimartingale whose characteristics are absolutely continuous in the time variable. Using \eqref{eq:B(h)}, we see that the fact of $X$ being an It\^o semimartingale does not depend on the choice of the truncation function. The notion of differential characteristics of It\^o semimartingales is a generalization of the L\'evy--Khintchine triplet of L\'evy processes. By the famous L\'evy--Khintchine formula (see e.g.\ \cite{MR2273672}, Theorem I.43), the differential characteristics of a L\'evy process consist of a vector $b \in \mathbb{R}^d$, a matrix $c \in \mathbb{S}^d_+$ and a L\'evy measure $\kappa$ on $\mathbb{R}^d$, which are deterministic and independent of time.

In the case where $X$ is a special semimartingale, there is a natural choice of the characteristics triplet which is defined in a truncation function-free way.

\begin{definition}
	Let $X$ be an $\mathbb{R}^d$-valued special semimartingale. Let $B$ be the predictable finite variation part of $X$. Let $C$ and $\nu$ be the second and third characteristics of $X$ respectively. We call the triplet $(B, C, \nu)$ the \emph{canonical characteristics} of $X$.
\end{definition}

Note that given a truncation function $h$, one can still talk about the characteristics $(B(h), C, \nu)$ of $X$ associated with $h$. Analogous to \eqref{eq:B(h)}, the relationship between $B$ and $B(h)$ is given by \cite{MR1943877}, Proposition~II.2.29(a):
\begin{equation}\label{eq:B_B(h)}
	B_t - B(h)_t
	= \int_{[0, t] \times \mathbb{R}^d} (\xi - h(\xi)) \,\nu(ds, d\xi).
\end{equation}
Also, similar to Definition~\ref{def:diff_charac}, we have the notion of \emph{canonical differential characteristics} for special It\^o semimartingales.

\subsection{Statement of the main results}
Now we are able to state our main results.

\begin{theorem}\label{thm:mp}
	Let $\mathcal{E}$ be a Polish space. Let $(\Omega, \mathcal{F}, (\mathcal{F}_t)_{t \geq 0}, \mathbb{P})$ be a filtered probability space, with right-continuous filtration, that supports an $\mathcal{E}$-valued $\mathcal{F}_0$-measurable random variable $Z_0$ and an $\mathbb{R}^d$-valued It\^o semimartingale $Y$ with $Y_0 = 0$ and characteristics triplet $(B, C, \nu)$ associated with a truncation function $h$:
	\begin{equation}\label{eq:char}
		B_t = \int_0^t b_s \,ds,\quad
		C_t = \int_0^t c_s \,ds,\quad
		\nu([0, t] \times A) = \int_0^t \kappa_s(A) \,ds,
	\end{equation}
	where $b$ is an $\mathbb{R}^d$-valued predictable process, $c$ is an $\mathbb{S}^d_+$-valued predictable process, and $\kappa$ is a predictable L\'evy transition kernel from $\Omega \times \mathbb{R}_+$ to $\mathbb{R}^d$. Suppose that $(b, c, \kappa)$ satisfy
	\begin{equation}\label{eq:thm_asm}
		\mathbb{E}\biggl[\int_0^t \biggl(|b_s| + |c_s| + \int_{\mathbb{R}^d} 1 \land |\xi|^2 \,\kappa_s(d\xi)\biggr) \,ds\biggr] < \infty,\quad
		\forall\, t > 0.
	\end{equation}
	Let $\Phi: \mathcal{E} \times D^d_0 \to D^{\mathcal{E}}$ be a continuous updating function, and let $Z = \Phi(Z_0, Y)$. Then, there exist measurable functions $\widehat{b}: \mathbb{R}_+ \times \mathcal{E} \to \mathbb{R}^d$, $\widehat{c}: \mathbb{R}_+ \times \mathcal{E} \to \mathbb{S}_+^d$, and a L\'evy transition kernel $\widehat{\kappa}$ from $\mathbb{R}_+ \times \mathcal{E}$ to $\mathbb{R}^d$ such that for Lebesgue-a.e.\ $t \geq 0$,
	\begin{equation}\label{eq:cond_exp}
		\begin{split}
			\widehat{b}(t, Z_t) &= \mathbb{E}[b_t \,|\, Z_t],\\
			\widehat{c}(t, Z_t) &= \mathbb{E}[c_t \,|\, Z_t],\\
			\int_A 1 \land |\xi|^2 \,\widehat{\kappa}(t, Z_t, d\xi)
			&= \mathbb{E}\biggl[\int_A 1 \land |\xi|^2 \,\kappa_t(d\xi) \,\bigg|\, Z_t\biggr],\quad
			\forall\, A \in \mathcal{B}(\mathbb{R}^d).
		\end{split}
	\end{equation}
	Furthermore, there exists a filtered probability space $(\widehat{\Omega}, \widehat{\mathcal{F}}, (\widehat{\mathcal{F}}_t)_{t \geq 0}, \widehat{\mathbb{P}})$, with right-continuous filtration, that supports an $\mathcal{E}$-valued $\widehat{\mathcal{F}}_0$-measurable random variable $\widehat{Z}_0$ and an $\mathbb{R}^d$-valued c\`adl\`ag process $\widehat{Y}$ with $\widehat{Y}_0 = 0$ such that:
	\begin{enumerate}[label=(\roman*), topsep=0.3em, noitemsep]
		\item $\widehat{Y}$ is an It\^o semimartingale with characteristics triplet $(\widehat{B}, \widehat{C}, \widehat{\nu})$ associated with $h$:
		\begin{equation}\label{eq:char_2}
			\widehat{B}_t = \int_0^t \widehat{b}(s, \widehat{Z}_s) \,ds,\quad
			\widehat{C}_t = \int_0^t \widehat{c}(s, \widehat{Z}_s) \,ds,\quad
			\widehat{\nu}([0, t] \times A) = \int_0^t \widehat{\kappa}(s, \widehat{Z}_s, A) \,ds,
		\end{equation}
		where $\widehat{Z} = \Phi(\widehat{Z}_0, \widehat{Y})$,
		
		\item for each $t \geq 0$, the law of $\widehat{Z}_t$ under $\widehat{\mathbb{P}}$ agrees with the law of $Z_t$ under $\mathbb{P}$.
	\end{enumerate}
\end{theorem}

\begin{remark}
	By \eqref{eq:B(h)}, it is easy to check that the integrability condition \eqref{eq:thm_asm} does not depend on the choice of the truncation function $h$. Also, as discussed in \cite{MR4814246}, Equation~(3.6), the third identity of \eqref{eq:cond_exp} is equivalent to the following: for Lebesgue-a.e.\ $t \geq 0$,
	\begin{equation}\label{eq:cond_exp_2}
		\int_{\mathbb{R}^d} f(\xi) \,\widehat{\kappa}(t, Z_t, d\xi)
		= \mathbb{E}\biggl[\int_{\mathbb{R}^d} f(\xi) \,\kappa_t(d\xi) \,\bigg|\, Z_t\biggr],
	\end{equation}
	for all measurable functions $f: \mathbb{R}^d \to \mathbb{R}$ satisfying $|f(\xi)| \leq C (1 \land |\xi|^2)$, $\forall\, \xi \in \mathbb{R}^d$, for some constant $C > 0$. Since $Z$ is a c\`adl\`ag process, we have $Z_t = Z_{t-}$ $\mathbb{P}$-a.s.\ for all but countably many $t \geq 0$. Thus, we may replace $Z_t$ by $Z_{t-}$ in \eqref{eq:cond_exp} and \eqref{eq:cond_exp_2}.
\end{remark}

\begin{remark}
	We emphasize again that a Markovian projection of an It\^o semimartingale is not guaranteed to be a Markov process. As discussed in the introduction, Markovian projections only assume Markovian-type dynamics. We give an example here where the resulting projection is Markov. Consider the simplest case in Example~\ref{eg:1} (process itself). If $Y$ is continuous (i.e.\ $\kappa = 0$) and the dimension $d \leq 2$, then the Markovian projection $\widehat{Y}$ is a true Markov process, as long as $b$, $c$ are bounded, and $c$ is uniformly elliptic. This follows from \eqref{eq:cond_exp} and the classical results of \cite{MR2190038}, Exercise~7.33 ($d=1$) and \cite{MR1483890}, Theorem~VI.4.5 ($d=2$). However, in more general cases, one has to impose extra regularity conditions on $(\widehat{b}, \widehat{c}, \widehat{\kappa})$ to have the Markov property, and we will further discuss it at the end of this section.
\end{remark}

To prove Theorem~\ref{thm:mp}, we adopt similar ideas from Brunick and Shreve~\cite{MR3098443}, which involve the construction of a canonical space, a discretization of time and passage to the limit. However, extra work is needed to deal with the third characteristic of a jump diffusion process. This requires a nontrivial extension of the canonical space from the continuous case. See Section~\ref{sec:4} for the detailed proof. We will also discuss the motivation for our choice of the canonical space at the end of this paper. See Remark~\ref{rem:canonical_space}.

The following corollary works with special semimartingales and their canonical characteristics. As we see in Theorem~\ref{thm:mp}, even if $Y$ is a special semimartingale, we still need to find some truncation function to deal with its large jumps. Sometimes this is unnecessary and complicates the computation, given that the large jumps already have nice integrability. Working with canonical characteristics would be more convenient in this case.

\begin{corollary}\label{cor:mp}
	Let $\mathcal{E}$ be a Polish space. Let $(\Omega, \mathcal{F}, (\mathcal{F}_t)_{t \geq 0}, \mathbb{P})$ be a filtered probability space, with right-continuous filtration, that supports an $\mathcal{E}$-valued $\mathcal{F}_0$-measurable random variable $Z_0$ and an $\mathbb{R}^d$-valued special It\^o semimartingale $Y$ with $Y_0 = 0$ and canonical characteristics triplet $(B, C, \nu)$:
	\begin{equation*}
		B_t = \int_0^t b_s \,ds,\quad
		C_t = \int_0^t c_s \,ds,\quad
		\nu([0, t] \times A) = \int_0^t \kappa_s(A) \,ds,
	\end{equation*}
	where $b$ is an $\mathbb{R}^d$-valued predictable process, $c$ is an $\mathbb{S}^d_+$-valued predictable process, and $\kappa$ is a predictable L\'evy transition kernel from $\Omega \times \mathbb{R}_+$ to $\mathbb{R}^d$. Suppose that $(b, c, \kappa)$ satisfy
	\begin{equation}\label{eq:cor_asm}
		\mathbb{E}\biggl[\int_0^t \biggl(|b_s| + |c_s| + \int_{\mathbb{R}^d} |\xi| \land |\xi|^2 \,\kappa_s(d\xi)\biggr) \,ds\biggr] < \infty,\quad
		\forall\, t > 0.
	\end{equation}
	Let $\Phi: \mathcal{E} \times D^d_0 \to D^{\mathcal{E}}$ be a continuous updating function, and let $Z = \Phi(Z_0, Y)$. Then, there exist measurable functions $\widehat{b}: \mathbb{R}_+ \times \mathcal{E} \to \mathbb{R}^d$, $\widehat{c}: \mathbb{R}_+ \times \mathcal{E} \to \mathbb{S}_+^d$, and a L\'evy transition kernel $\widehat{\kappa}$ from $\mathbb{R}_+ \times \mathcal{E}$ to $\mathbb{R}^d$ such that for Lebesgue-a.e.\ $t \geq 0$,
	\begin{equation}\label{eq:cor_cond_exp}
		\begin{split}
			\widehat{b}(t, Z_t) &= \mathbb{E}[b_t \,|\, Z_t],\\
			\widehat{c}(t, Z_t) &= \mathbb{E}[c_t \,|\, Z_t],\\
			\int_A |\xi| \land |\xi|^2 \,\widehat{\kappa}(t, Z_t, d\xi)
			&= \mathbb{E}\biggl[\int_A |\xi| \land |\xi|^2 \,\kappa_t(d\xi) \,\bigg|\, Z_t\biggr],\quad
			\forall\, A \in \mathcal{B}(\mathbb{R}^d).
		\end{split}
	\end{equation}
	Furthermore, there exists a filtered probability space $(\widehat{\Omega}, \widehat{\mathcal{F}}, (\widehat{\mathcal{F}}_t)_{t \geq 0}, \widehat{\mathbb{P}})$, with right-continuous filtration, that supports an $\mathcal{E}$-valued $\widehat{\mathcal{F}}_0$-measurable random variable $\widehat{Z}_0$ and an $\mathbb{R}^d$-valued c\`adl\`ag process $\widehat{Y}$ with $\widehat{Y}_0 = 0$ such that:
	\begin{enumerate}[label=(\roman*), topsep=0.3em, noitemsep]
		\item $\widehat{Y}$ is a special It\^o semimartingale with canonical characteristics triplet $(\widehat{B}, \widehat{C}, \widehat{\nu})$:
		\begin{equation*}
			\widehat{B}_t = \int_0^t \widehat{b}(s, \widehat{Z}_s) \,ds,\quad
			\widehat{C}_t = \int_0^t \widehat{c}(s, \widehat{Z}_s) \,ds,\quad
			\widehat{\nu}([0, t] \times A) = \int_0^t \widehat{\kappa}(s, \widehat{Z}_s, A) \,ds,
		\end{equation*}
		where $\widehat{Z} = \Phi(\widehat{Z}_0, \widehat{Y})$,
		
		\item for each $t \geq 0$, the law of $\widehat{Z}_t$ under $\widehat{\mathbb{P}}$ agrees with the law of $Z_t$ under $\mathbb{P}$.
	\end{enumerate}
\end{corollary}

\begin{remark}
	Note that if $Y$ is an It\^o semimartingale (a priori not special) with differential characteristics $(b(h),c,\kappa)$ associated with some truncation function $h$, and the triplet satisfies \eqref{eq:cor_asm}, then $Y$ is necessarily a special semimartingale and its canonical differential characteristics also satisfy \eqref{eq:cor_asm}. This follows from \eqref{eq:B_B(h)}.
	
	Conversely, if $Y$ is already a special semimartingale but $(b(h),c,\kappa)$ only satisfy \eqref{eq:thm_asm}, we still cannot apply Theorem~\ref{thm:mp} to its canonical characteristics. To use Corollary~\ref{cor:mp} which does not involve any truncation functions, we have to pay the price of a stronger integrability condition \eqref{eq:cor_asm}.
	
	Finally, as is in Theorem~\ref{thm:mp}, we may replace $Z_t$ by $Z_{t-}$ in \eqref{eq:cor_cond_exp}.
\end{remark}

\begin{proof}[Proof of Corollary~\ref{cor:mp}]
	Let $h: \mathbb{R}^d \to \mathbb{R}^d$ be a truncation function, and let $(B^h, C, \nu)$ be the characteristics of $Y$ associated with $h$. Denote
	\begin{equation*}
		B^h_t = \int_0^t b^h_s \,ds,
	\end{equation*}
	where $b^h$ is an $\mathbb{R}^d$-valued predictable process. By \eqref{eq:B_B(h)} and \eqref{eq:cor_asm}, it is easy to check that $(b^h, c, \kappa)$ satisfy \eqref{eq:thm_asm}. Then, Theorem~\ref{thm:mp} yields measurable functions $\widehat{b}^h: \mathbb{R}_+ \times \mathcal{E} \to \mathbb{R}^d$, $\widehat{c}: \mathbb{R}_+ \times \mathcal{E} \to \mathbb{S}_+^d$, and a L\'evy transition kernel $\widehat{\kappa}$ from $\mathbb{R}_+ \times \mathcal{E}$ to $\mathbb{R}^d$ such that $(\widehat{c}, \widehat{\kappa})$ and $(c, \kappa)$ satisfy \eqref{eq:cond_exp}, thus \eqref{eq:cor_cond_exp} by approximation,\footnote{By approximation, \eqref{eq:cond_exp} implies $\int_{\mathbb{R}^d} f(\xi) (1 \land |\xi|^2) \,\widehat{\kappa}(t, Z_t, d\xi) = \mathbb{E}[\int_{\mathbb{R}^d} f(\xi) (1 \land |\xi|^2) \,\kappa_t(d\xi) \,|\, Z_t]$ for all nonnegative measurable $f$. By taking $f(\xi) = \bm{1}_A(\xi) (|\xi| \land |\xi|^2) / (1 \land |\xi|^2)$, we get \eqref{eq:cor_cond_exp}.} and for Lebesgue-a.e.\ $t \geq 0$,
	\begin{equation*}
		\widehat{b}^h(t, Z_t)
		= \mathbb{E}[b^h_t \,|\, Z_t].
	\end{equation*}
	If we take a closer look at the construction of $\widehat{\kappa}$, which is based on \cite{MR4814246}, Lemma~2.5 (with an obvious extension to $\mathcal{E}$-valued processes and transition kernels from $\mathbb{R}_+ \times \mathcal{E}$ to $\mathbb{R}^d$), one may require $\widehat{\kappa}$ to satisfy the following property due to \eqref{eq:cor_asm}:
	\begin{equation*}
		\int_{\mathbb{R}^d} |\xi| \land |\xi|^2 \,\widehat{\kappa}(t, z, d\xi) < \infty,\quad
		\forall\, t \geq 0,\, z \in \mathcal{E}.
	\end{equation*}
	Moreover, Theorem~\ref{thm:mp} yields a filtered probability space $(\widehat{\Omega}, \widehat{\mathcal{F}}, (\widehat{\mathcal{F}}_t)_{t \geq 0}, \widehat{\mathbb{P}})$, with right-continuous filtration, that supports an $\mathcal{E}$-valued random variable $\widehat{Z}_0$ and an $\mathbb{R}^d$-valued It\^o semimartingale $\widehat{Y}$ with initial value $0$ and characteristics $(\widehat{B}^h, \widehat{C}, \widehat{\nu})$ associated with $h$, such that the one-dimensional marginal laws of $\widehat{Z} = \Phi(\widehat{Z}_0, \widehat{Y})$ agree with $Z$. Here $\widehat{C}$ and $\widehat{\nu}$ are defined as in \eqref{eq:char_2}, and
	\begin{equation*}
		\widehat{B}^h_t = \int_0^t \widehat{b}^h(s, \widehat{Z}_s) \,ds.
	\end{equation*}
	
	It only remains to show $\widehat{Y}$ is a special semimartingale with the desired canonical characteristics. From \eqref{eq:cor_asm}, \eqref{eq:cor_cond_exp} and the fact that $\widehat{Z}$ and $Z$ have the same one-dimensional marginal laws, we get
	\begin{equation*}
		\widehat{\mathbb{E}}\biggl[\int_0^t \int_{\mathbb{R}^d} |\xi| \land |\xi|^2 \,\widehat{\kappa}(s, \widehat{Z}_s, d\xi) \,ds\biggr]
		= \mathbb{E}\biggl[\int_0^t \int_{\mathbb{R}^d} |\xi| \land |\xi|^2 \,\widehat{\kappa}(s, Z_s, d\xi) \,ds\biggr] < \infty,\quad
		\forall\, t \geq 0.
	\end{equation*}
	Thus, \cite{MR1943877}, Proposition~II.2.29(a) implies that $\widehat{Y}$ is a special semimartingale. We define the measurable function $\widehat{b}: \mathbb{R}_+ \times \mathcal{E} \to \mathbb{R}^d$ via
	\begin{equation*}
		\widehat{b}(t, z)
		\coloneqq \widehat{b}^h(t, z) + \int_{\mathbb{R}^d} (\xi - h(\xi)) \,\widehat{\kappa}(t, z, d\xi),\quad
		t \geq 0,\, z \in \mathcal{E},
	\end{equation*}
	which we will show to be independent of the choice of $h$. Then by \eqref{eq:B_B(h)}, the first canonical characteristic $\widehat{B}$ of $\widehat{Y}$ is given by
	\begin{equation*}
		\begin{split}
			\widehat{B}_t
			&= \widehat{B}^h_t + \int_{[0, t] \times \mathbb{R}^d} (\xi - h(\xi)) \,\widehat{\nu}(ds, d\xi)\\
			&= \int_0^t \widehat{b}^h(s, \widehat{Z}_s) \,ds + \int_0^t \int_{\mathbb{R}^d} (\xi - h(\xi)) \,\widehat{\kappa}(s, \widehat{Z}_s, d\xi) \,ds
			= \int_0^t \widehat{b}(s, \widehat{Z}_s) \,ds.
		\end{split}
	\end{equation*}
	Another application of \eqref{eq:B_B(h)} shows that for Lebesgue-a.e.\ $t \geq 0$,
	\begin{equation*}
		\begin{split}
			\widehat{b}(t, Z_t)
			&= \widehat{b}^h(t, Z_t) + \int_{\mathbb{R}^d} (\xi - h(\xi)) \,\widehat{\kappa}(t, Z_t, d\xi)\\
			&= \mathbb{E}[b^h_t \,|\, Z_t] + \mathbb{E}\biggl[\int_{\mathbb{R}^d} (\xi - h(\xi)) \,\kappa_t(d\xi) \,\bigg|\, Z_t\biggr]
			= \mathbb{E}[b_t \,|\, Z_t].
		\end{split}
	\end{equation*}
	This finishes the proof.
\end{proof}

The following example is an application of Corollary~\ref{cor:mp} and shows that Markovian projections preserve iterated integral structures.

\begin{example}
	Let $X$ be a real-valued special It\^o semimartingale with $X_0 = 0$ and canonical differential characteristics $(b, c, \kappa)$. Let $Y = \int_0^\cdot X_{s-} \,dX_s$, which is also a special It\^o semimartingale. One can write $X$ and $Y$ in the form of iterated integrals: $X_t = \int_0^t dX_s$ and $Y_t = \int_0^t \int_0^{s-} dX_u \,dX_s$. Note that $\Delta Y = X_- \Delta X$, so we have (recall Definition~\ref{def:char}(iii))
	\begin{equation*}
		\mu^{(X, Y)} = \mu^X \circ ((t, \xi) \mapsto (t, \xi, X_{t-} \xi))^{-1},
	\end{equation*}
	and we can easily compute the canonical differential characteristics of $(X, Y)$:
	\begin{equation*}
		(1, X_{t-}) b_t,\quad
		\begin{pmatrix}
			1      & X_{t-}\\
			X_{t-} & X_{t-}^2
		\end{pmatrix} c_t,\quad
		\kappa_t \circ (\xi \mapsto (\xi, X_{t-} \xi))^{-1}.
	\end{equation*}
	Assume that
	\begin{equation*}
		\begin{split}
			\mathbb{E}\biggl[\int_0^t \biggl(&(1 + |X_s|) |b_s| + (1 + |X_s|^2) |c_s|\\ &+ \int_{\mathbb{R}} (1 + |X_s|) |\xi| \land (1 + |X_s|^2) |\xi|^2 \,\kappa_s(d\xi)\biggr) \,ds\biggr] < \infty,\quad
			\forall\, t > 0.
		\end{split}
	\end{equation*}
	Then, applying Corollary~\ref{cor:mp} to the special It\^o semimartingale $(X, Y)$, the initial value $Z_0 = (0, 0)$ and the updating function $\Phi(e, x) = e + x$ for $e \in \mathbb{R}^2$, $x \in D^2_0$, we know that the Markovian projection $(\widehat{X}, \widehat{Y})$ is a special It\^o semimartingale whose canonical differential characteristics have the form
	\begin{equation*}
		(1, \widehat{X}_{t-}) \widehat{b}(t, \widehat{X}_{t-}, \widehat{Y}_{t-}),\quad
		\begin{pmatrix}
			1                & \widehat{X}_{t-}\\
			\widehat{X}_{t-} & \widehat{X}_{t-}^2
		\end{pmatrix} \widehat{c}(t, \widehat{X}_{t-}, \widehat{Y}_{t-}),\quad
		\widehat{\kappa}(t, \widehat{X}_{t-}, \widehat{Y}_{t-}, \cdot) \circ (\xi \mapsto (\xi, \widehat{X}_{t-} \xi))^{-1}.
	\end{equation*}
	It follows that
	\begin{equation*}
		\begin{split}
			&\widehat{\mathbb{E}}\Bigl[\mu^{(\widehat{X}, \widehat{Y})}(\{(t, \xi, \eta) \in \mathbb{R}_+ \times \mathbb{R}^2: \eta \neq \widehat{X}_{t-} \xi\})\Bigr]\\
			&\quad\quad\quad\quad= \widehat{\mathbb{E}}\biggl[\int_0^\infty \int_{\mathbb{R}} \bm{1}_{\{(x, y) \in \mathbb{R}^2: y \neq \widehat{X}_{t-} x\}}(\xi, \widehat{X}_{t-} \xi) \,\widehat{\kappa}(t, \widehat{X}_{t-}, \widehat{Y}_{t-}, d\xi) \,dt\biggr]
			= 0,
		\end{split}
	\end{equation*}
	i.e.\ $\mu^{(\widehat{X}, \widehat{Y})}(\{(t, \xi, \eta) \in \mathbb{R}_+ \times \mathbb{R}^2: \eta \neq \widehat{X}_{t-} \xi\}) = 0$ $\widehat{\mathbb{P}}$-a.s. This implies that $\Delta \widehat{Y} = \widehat{X}_- \Delta \widehat{X}$ $\widehat{\mathbb{P}}$-a.s. Thus, by \cite{MR1943877}, Corollary~II.2.38, we have the canonical decomposition of $\widehat{X}$ and $\widehat{Y}$:
	\begin{equation*}
		\begin{split}
			\widehat{X}_t &= \int_0^t \widehat{b}(s, \widehat{X}_{s-}, \widehat{Y}_{s-}) \,ds + \widehat{X}^c_t + \int_{[0, t] \times \mathbb{R}} \xi \,(\mu^{\widehat{X}}(ds, d\xi) - \widehat{\kappa}(s, \widehat{X}_{s-}, \widehat{Y}_{s-}, d\xi)ds),\\
			\widehat{Y}_t &= \int_0^t \widehat{X}_{s-} \widehat{b}(s, \widehat{X}_{s-}, \widehat{Y}_{s-}) \,ds + \widehat{Y}^c_t + \int_{[0, t] \times \mathbb{R}} \widehat{X}_{s-} \xi \,(\mu^{\widehat{X}}(ds, d\xi) - \widehat{\kappa}(s, \widehat{X}_{s-}, \widehat{Y}_{s-}, d\xi)ds).
		\end{split}
	\end{equation*}
	The predictable finite variation term is straightforward. For the continuous local martingale term, note that $\langle \widehat{Y}^c - \widehat{X}_- \cdot \widehat{X}^c, \widehat{Y}^c - \widehat{X}_- \cdot \widehat{X}^c \rangle$ = 0, so we have $\widehat{Y}^c = \widehat{X}_- \cdot \widehat{X}^c$ $\widehat{\mathbb{P}}$-a.s. For the purely discontinuous local martingale term, we use \cite{MR1943877}, Proposition~II.1.30(b). From these we conclude that $\widehat{Y} = \int_0^\cdot \widehat{X}_{s-} \,d\widehat{X}_s$, so the iterated integral structure is preserved. This example can be easily extended to $\mathbb{R}^d$-valued special It\^o semimartingales $X$, and higher order iterated integrals.
\end{example}

In practice, when we want to apply Theorem~\ref{thm:mp} or Corollary~\ref{cor:mp} to an It\^o semimartingale $Y$, it is very common that $Y$ comes from an SDE possibly with some exogenous random sources, or equivalently, $Y$ is a component of a solution to a larger system of SDEs. The following example gives a class of such processes $Y$ that satisfy the assumption of our main results, so we can talk about their Markovian projections. To avoid the extra terms introduced by truncation functions and make our presentation concise, we consider the case of special semimartingales.

\begin{example}
	Let $(Y, Y^\prime)$ be an $(\mathbb{R}^d \times \mathbb{R}^{d^\prime})$-valued special It\^o semimartingale with $(Y_0, Y^\prime_0) = (0, 0)$ and canonical differential characteristics of the form\footnote{It suffices to assume time-homogeneity since we can always absorb the time variable $t$ into the process $Y^\prime$.}
	\begin{equation*}
		B(Y_t, Y^\prime_t),\quad
		C(Y_t, Y^\prime_t),\quad
		\mathrm{K}(Y_t, Y^\prime_t, d\xi, d\xi^\prime),
	\end{equation*}
	where $B: \mathbb{R}^{d + d^\prime} \to \mathbb{R}^{d + d^\prime}$, $C: \mathbb{R}^{d + d^\prime} \to \mathbb{S}_+^{d + d^\prime}$ are measurable functions, and $\mathrm{K}$ is a L\'evy transition kernel from $\mathbb{R}^{d + d^\prime}$ to $\mathbb{R}^{d + d^\prime}$. In other words, $(Y, Y^\prime)$ jointly solve a Markovian SDE
	\begin{equation*}
		\begin{split}
			(Y_t, Y^\prime_t) = \int_0^t B(Y_s, Y^\prime_s) \,ds &+ \int_0^t C^{1/2}(Y_{s-}, Y^\prime_{s-}) \,dW_s\\ &+ \int_{[0, t] \times \mathbb{R}^{d + d^\prime}} (\xi, \xi^\prime) \,(\mu^{(Y, Y^\prime)}(ds, d\xi, d\xi^\prime) - \mathrm{K}(Y_s, Y^\prime_s, d\xi, d\xi^\prime)ds),
		\end{split}
	\end{equation*}
	where $C^{1/2}$ is the symmetric square-root of $C$, and $W$ is a $(d + d^\prime)$-dimensional Brownian motion. Suppose that $(B, C, \mathrm{K})$ satisfy
	\begin{equation}\label{eq:quad_growth}
		\begin{split}
			|B(y, y^\prime)|^2 + |C(y, y^\prime)| + \int_{\mathbb{R}^{d + d^\prime}} (|\xi|^2 + |\xi^\prime|^2) \,\mathrm{K}(y, y^\prime, d\xi, d\xi^\prime)
			\leq C(1 + |y|^2 + |y^\prime|^2),\\
			\forall\, (y, y^\prime) \in \mathbb{R}^d \times \mathbb{R}^{d^\prime},
		\end{split}
	\end{equation}
	for some constant $C > 0$. Then, by the Burkholder--Davis--Gundy inequality, and \cite{MR1943877}, Theorem~I.4.52, II.1.8, we have the following estimate
	\begin{equation*}
		\begin{split}
			&\mathbb{E}[|Y_t|^2 + |Y^\prime_t|^2]
			\leq C_1\mathbb{E}\Biggl[\biggl|\int_0^t B(Y_s, Y^\prime_s) \,ds\biggr|^2 + \int_0^t |\mathrm{tr}(C(Y_s, Y^\prime_s))| \,ds + \sum_{s \leq t} (|Y_s|^2 + |Y^\prime_s|^2)\Biggr]\\
			&\quad\quad\quad\leq C_2 \mathbb{E}\biggl[\int_0^t \biggl(t|B(Y_s, Y^\prime_s)|^2 + |C(Y_s, Y^\prime_s)| + \int_{\mathbb{R}^{d + d^\prime}} (|\xi|^2 + |\xi^\prime|^2) \,\mathrm{K}(Y_s, Y^\prime_s, d\xi, d\xi^\prime)\biggr) \,ds\biggr],
		\end{split}
	\end{equation*}
	for some constants $C_1, C_2 > 0$. Using assumption \eqref{eq:quad_growth} and a Gr\"onwall-type of argument, it is easy to check that $t \mapsto \mathbb{E}[|Y_t|^2 + |Y^\prime_t|^2]$ is locally bounded. Thus, $Y$ satisfies \eqref{eq:cor_asm} with
	\begin{equation*}
		b_t = b(Y_t, Y^\prime_t),\quad
		c_t = c(Y_t, Y^\prime_t),\quad
		\kappa_t(d\xi) = \kappa(Y_t, Y^\prime_t, d\xi),
	\end{equation*}
	where $b$ is the first $d$ coordinates of $B$, $c$ is the upper left $d \times d$ submatrix of $C$, and $\kappa(y, y^\prime, d\xi)$ is the push forward of $\mathrm{K}(y, y^\prime, d\xi, d\xi^\prime)$ under the projection map $(\xi, \xi^\prime) \mapsto \xi$ for each $(y, y^\prime) \in \mathbb{R}^d \times \mathbb{R}^{d^\prime}$.
\end{example}

As mentioned in the introduction, the assumption \eqref{eq:thm_asm} of our main results is weaker than those in existing literature. We next present an example which fits our framework, but does not satisfy \cite{MR4814246}, Equation~(3.4).

\begin{example}
	Let $Y$ be a real-valued doubly stochastic compound Poisson process with $Y_0 = 0$, intensity process $\lambda$ and jump size distribution $\nu$. That is to say, $Y$ is of the form $Y_t = \sum_{k=1}^\infty J_k \bm{1}_{\{\tau_k \leq t\}}$, where $\tau_0 \coloneqq 0$, $\tau_k \coloneqq \inf\bigl\{t > \tau_{k-1}: \int_{\tau_{k-1}}^t \lambda_s \,ds \geq E_k\bigr\}$ for $k \geq 1$, $(E_k)_{k \in \mathbb{N}^*}$ is a sequence of i.i.d.\ $\mathrm{Exp}(1)$ random variables, and $(J_k)_{k \in \mathbb{N}^*}$ is a sequence of i.i.d.\ $\nu$-distributed random variables, such that $\lambda$, $(E_k)_{k \in \mathbb{N}^*}$ and $(J_k)_{k \in \mathbb{N}^*}$ are independent. Suppose that $1 \leq \lambda \leq 2$ $(\mathbb{P} \otimes dt)$-a.e.\ and
	\begin{equation*}
		\nu(\{0\}) = 0,\quad
		\int_{\mathbb{R}} 1 \land |\xi|^2 \,\nu(d\xi) < \infty,\quad
		\int_{\{|\xi| > 1\}} \log(1 + |\xi|) \,\nu(d\xi) = \infty.
	\end{equation*}
	It is a standard fact (see e.g.\ \cite{MR4976813}, Theorem~2.14) that $Y$ is an It\^o semimartingale whose differential characteristics associated with a truncation function $h: \mathbb{R} \to \mathbb{R}$ are
	\begin{equation*}
		b_t = \mathbb{E}[h(J_1)] \lambda_t,\quad
		c_t = 0,\quad
		\kappa_t(d\xi) = \lambda_t \nu(d\xi).
	\end{equation*}
	Then, it is easy to see that $(b, c, \kappa)$ satisfy \eqref{eq:thm_asm}, so we can apply Theorem~\ref{thm:mp} to $Y$, say with the initial value $Z_0 = 0$ and the updating function $\Phi(e, x) = e + x$ for $e \in \mathbb{R}$, $x \in D^1_0$. However, taking conditional expectation $\mathbb{E}[\cdot \,|\, Y_t]$, we have
	\begin{equation*}
		\widehat{b}(t, y) = \mathbb{E}[h(J_1)] \widehat{\lambda}(t, y),\quad
		\widehat{c}(t, y) = 0,\quad
		\widehat{\kappa}(t, y, d\xi) = \widehat{\lambda}(t, y) \nu(d\xi),
	\end{equation*}
	where $\widehat{\lambda}: \mathbb{R}_+ \times \mathbb{R} \to [1, 2]$ is a measurable function such that $\mathbb{E}[\lambda_t \,|\, Y_t] = \widehat{\lambda}(t, Y_t)$ for Lebesgue-a.e.\ $t \geq 0$. Since $\log(1 + \lvert\cdot\rvert)$ is not integrable with respect to $\nu$ around infinity, it follows that $(\widehat{b}, \widehat{c}, \widehat{\kappa})$ do not satisfy \cite{MR4814246}, Equation~(3.4) for any truncation function of the form $h(x) = x \bm{1}_{\{|x| \leq r\}}$ with $r > 0$, thus Theorem~3.2 therein does not apply to $Y$.
\end{example}

We end this section by making some comments on the uniqueness in law and the Markov property of the mimicking process. In Theorem~\ref{thm:mp}, we see from the specific forms of the characteristics of $\widehat{Y}$ and the definition of $\Phi$ that the mimicking process $\widehat{Z}$ has Markovian-type dynamics. Even with the simplest updating function given by Example~\ref{eg:1} (process itself), there is no guarantee that $\widehat{Z}$ is a true Markov process, or is unique in law, unless we are willing to impose further regularity assumptions on $\widehat{b}$, $\widehat{c}$ and $\widehat{\kappa}$. In practice, we use \eqref{eq:cond_exp} to compute these coefficients. If there are sufficiently ``nice'' versions, then properties beyond existence may hold. See e.g.\ \cite{bentata2012mimicking}, \cite{MR359017}, \cite{MR838085}, \cite{MR1248747}, \cite{MR1297694} for various conditions on non-local generators that imply the uniqueness and/or the Markov property of martingale solutions. For example, \cite{MR359017}, Theorem~4.2 requires the following sufficient condition:
\begin{enumerate}[label=(\roman*), topsep=0.3em, noitemsep]
	\item $\widehat{b}(t, x)$ is locally bounded,
	
	\item $\widehat{c}(t, x)$ is locally bounded, locally uniformly elliptic, and continuous in $x$ uniformly in $t \leq T$ for any $T > 0$,
	
	\item $\widehat{\kappa}(t, x, d\xi)$ is dominated by some L\'evy measure uniformly in $(t, x) \in D$ for any bounded set $D \subset \mathbb{R}_+ \times \mathbb{R}^d$.
\end{enumerate}
However, it is not straightforward to impose assumptions only on the differential characteristics $(b, c, \kappa)$ of the original process that translate to desired regularity conditions on the coefficients $(\widehat{b}, \widehat{c}, \widehat{\kappa})$ of the mimicking process, not to mention the possibly complicated structure of the updating function $\Phi$. Thus, in this paper we focus on existence results in general settings. See also \cite{MR3098443}, Section~3 for a discussion on the uniqueness and the Markov property.

\section{Preliminary Results}\label{sec:3}
In this section we introduce some notation and preliminary results needed for proving our main theorem. Some of the contents here are analogous to those in Brunick and Shreve~\cite{MR3098443}.

\subsection{Canonical Space}\label{sec:canonical}
Let $\mathcal{E}$ be a Polish space. Recall that in Section~\ref{sec:upd_fn}, we defined the difference operator $\Delta: D^{\mathcal{E}} \times \mathbb{R_+} \to D^{\mathcal{E}}_0$ if $\mathcal{E}$ is also a vector space. However, when $\mathcal{E}$ is not a vector space, the difference operator $\Delta$ may not be well-defined, either because the expression $\Delta(x, t)(s) = x(t + s) - x(t)$ has no meaning, or it is defined but does not belong to $\mathcal{E}$ (e.g.\ when $\mathcal{E}$ is a subset of a vector space). This leads to the following definition.

\begin{definition}\label{def:d_invar}
	Let $\mathcal{E}$ be a Polish space, which is also a subset of a vector space with $0 \in \mathcal{E}$. We say a subset $\mathcal{X}$ of $D^{\mathcal{E}}_0$ is $\Delta$-\emph{stable}, if $\Delta(x, t) \in \mathcal{X}$ for all $x \in \mathcal{X}$ and $t \geq 0$.
\end{definition}

Now we fix two Polish spaces $\mathcal{E}$ and $\mathcal{E}^\prime$, where $\mathcal{E}^\prime$ is a closed convex cone in a topological vector space.\footnote{Note that the topological vector space containing $\mathcal{E}^\prime$ does not need to be a Polish space. We do not even assume that its topology is metrizable.} We fix a space $\mathcal{X}$ which is a $\Delta$-stable closed subset of $D^{\mathcal{E}^\prime}_0$. By Definition~\ref{def:d_invar}, the map $\Delta: \mathcal{X} \times \mathbb{R}_+ \to \mathcal{X}$ is well-defined. Our canonical space is defined by $\Omega^{\mathcal{E}, \mathcal{X}} \coloneqq \mathcal{E} \times \mathcal{X}$. We endow $\Omega^{\mathcal{E}, \mathcal{X}}$ with the product topology, and we know $\Omega^{\mathcal{E}, \mathcal{X}}$ is a Polish space. For a generic element $\omega \in \Omega^{\mathcal{E}, \mathcal{X}}$, we denote it by $\omega = (e, x)$, and we define the projections
\begin{equation*}
	E(\omega) \coloneqq e,\quad
	X(\omega) \coloneqq x.
\end{equation*}
Let $\mathcal{F}^{\mathcal{E}, \mathcal{X}} = \sigma(E, X)$ be the Borel $\sigma$-algebra on $\Omega^{\mathcal{E}, \mathcal{X}}$, which can also be equivalently defined by $\sigma(E, X_t; t \geq 0)$. We also define the natural filtration $(\mathcal{F}^{\mathcal{E}, \mathcal{X}}_t)_{t \geq 0}$ generated by $E$ and $X$ via $\mathcal{F}^{\mathcal{E}, \mathcal{X}}_t \coloneqq \sigma(E, X^t) = \sigma(E, X_s; 0 \leq s \leq t)$.

We give three examples of spaces which are $\Delta$-stable closed subsets of Skorokhod spaces of the form $D^{\mathcal{E}^\prime}_0$. These examples will be used to construct our canonical space when proving the main theorem.

\begin{example}
	Let $\mathcal{E}^\prime = \mathbb{R}^d$. The space $D^d_0$ is trivially a $\Delta$-stable closed subset of $D^d_0$ itself. Suppose that $Y$ is an $\mathbb{R}^d$-valued semimartingale. Then, the sample paths of $Y - Y_0$ belong to $D^d_0$.
\end{example}

\begin{example}
	Let $\mathcal{E}^\prime = \mathbb{R}^d$. Recall that $C^d$ is a closed subspace of $D^d$ in the Skorokhod topology. Also, the Skorokhod topology in $D^d$ restricted to $C^d$ coincides with the topology of uniform convergence on compact intervals in $C^d$. Consequently, the space $C^d_0$ is a $\Delta$-stable closed subset of $D^d_0$. Suppose that $Y$ is an $\mathbb{R}^d$-valued It\^o semimartingale whose first two characteristics (associated with some truncation function) are given by $B = \int_0^\cdot b_s \,ds$ and $C = \int_0^\cdot c_s \,ds$, where $b$ is an $\mathbb{R}^d$-valued predictable process and $c$ is an $\mathbb{S}^d_+$-valued predictable process. Then, the sample paths of $B$ belong to $C^d_0$, and the sample paths of $C$ belong to $C^{d^2}_0$.
\end{example}

\begin{example}
	Let $\mathcal{E}^\prime = \mathcal{M}_+(\mathbb{R}^d)$ be the space of finite positive Borel measures on $\mathbb{R}^d$, $\mathcal{V} = \mathcal{M}(\mathbb{R}^d)$ be the space of finite signed Borel measures on $\mathbb{R}^d$, both endowed with the topology of weak convergence. It is well-known that $\mathcal{E}^\prime$ is a Polish space, while $\mathcal{V}$ is a topological vector space that is not metrizable. For simplicity, we write $C^{\mathcal{M}_+, d}_0$ (resp.\ $D^{\mathcal{M}_+, d}_0$) rather than $C^{\mathcal{M}_+(\mathbb{R}^d)}_0$ (resp.\ $D^{\mathcal{M}_+(\mathbb{R}^d)}_0$). We denote $C^{\mathcal{M}_+, d}_{0, \mathrm{i}}$ as the subset of $C^{\mathcal{M}_+, d}_0$ consisting of nondecreasing trajectories. Here we say a measure-valued function $\mu$ is nondecreasing, if $\mu_t - \mu_s$ is a positive measure for all $0 \leq s \leq t$. It is straightforward to check that $C^{\mathcal{M}_+, d}_{0, \mathrm{i}}$ is a $\Delta$-stable closed subset of $D^{\mathcal{M}_+, d}_0$. Suppose that $Y$ is an $\mathbb{R}^d$-valued It\^o semimartingale whose third characteristic is given by $\nu(dt, d\xi) = \kappa_t(d\xi)dt$, where $\kappa$ is a predictable L\'evy transition kernel from $\Omega \times \mathbb{R}_+$ to $\mathbb{R}^d$. If we define the measure-valued process $M$ via
	\begin{equation*}
		M_t(A)
		\coloneqq \int_0^t \int_A 1 \land |\xi|^2 \,\kappa_s(d\xi) \,ds,\quad
		t \geq 0,\, A \in \mathcal{B}(\mathbb{R}^d),
	\end{equation*}
	then the sample paths of $M$ belong to $C^{\mathcal{M}_+, d}_{0, \mathrm{i}}$.
\end{example}

\subsection{Concatenated probability measure}
Throughout this subsection, we fix a canonical space $\Omega^{\mathcal{E}, \mathcal{X}} = \mathcal{E} \times \mathcal{X}$, where $\mathcal{X}$ is a $\Delta$-stable closed subset of $D^{\mathcal{E}^\prime}_0$, $\mathcal{E}$ and $\mathcal{E}^\prime$ are Polish spaces, and $\mathcal{E}^\prime$ is also a closed convex cone in a topological vector space. Recall the notation $E$, $X$, $\mathcal{F}^{\mathcal{E}, \mathcal{X}}$ and $\mathcal{F}^{\mathcal{E}, \mathcal{X}}_t$ defined in Section~\ref{sec:canonical}. The contents in this subsection are similar to \cite{MR3098443}, Section~4, including the proofs. Since our canonical space $\Omega^{\mathcal{E}, \mathcal{X}}$ is more general than theirs, we will rephrase some results in \cite{MR3098443}.

\begin{definition}[cf.\ \cite{MR3098443}, Definition~4.1]\label{def:ext_part}
	Let $0 = T_0 \leq T_1 \leq \cdots \leq T_n < \infty$ be a sequence of finite stopping times on $(\Omega^{\mathcal{E}, \mathcal{X}}, \mathcal{F}^{\mathcal{E}, \mathcal{X}}, (\mathcal{F}^{\mathcal{E}, \mathcal{X}}_t)_{t \geq 0})$. Let $(\mathcal{G}_i)_{i=0}^n$ be a collection of $\sigma$-algebras satisfying $\mathcal{G}_i \subseteq \mathcal{F}^{\mathcal{E}, \mathcal{X}}_{T_i}$, $i = 0, ..., n$. Set $T_{n+1} \coloneqq \infty$, $\mathcal{H}_0 \coloneqq \mathcal{F}^{\mathcal{E}, \mathcal{X}}_0$, and define $\mathcal{H}_i \coloneqq \mathcal{G}_{i-1} \lor \sigma(\Delta(X^{T_i}, T_{i-1}))$, $i = 1, ..., n+1$. We say $\Pi = (T_i, \mathcal{G}_i)_{i=0}^n$ is an \emph{extended partition} if
	\begin{enumerate}[label=(\roman*), topsep=0.3em, noitemsep]
		\item $T_{i+1} - T_i$ is $\mathcal{H}_{i+1}$-measurable, $i = 0, ..., n-1$,
		
		\item $\mathcal{G}_i \subseteq \mathcal{H}_i$, $i = 0, ..., n$.
	\end{enumerate}
\end{definition}

Intuitively speaking, an extended partition is a model for keeping partial information over time. At time $T_i$, our information set is $\mathcal{H}_i$. We only keep $\mathcal{G}_i$ and forget everything else. Then at time $T_{i+1}$, we gain new information through the increment of $X$ on $[T_i, T_{i+1}]$, so our information set now becomes $\mathcal{H}_{i+1}$.

Next we equip $(\Omega^{\mathcal{E}, \mathcal{X}}, \mathcal{F}^{\mathcal{E}, \mathcal{X}})$ with a probability measure $\mathbb{P}$. We construct another probability measure $\mathbb{P}^{\otimes \Pi}$, called the \emph{concatenated measure}, based on an extended partition $\Pi$.

\begin{theorem}[cf.\ \cite{MR3098443}, Theorem~4.3]\label{thm:concat_meas}
	Let $\mathbb{P}$ be a probability measure on $(\Omega^{\mathcal{E}, \mathcal{X}}, \mathcal{F}^{\mathcal{E}, \mathcal{X}})$ and let $\Pi = (T_i, \mathcal{G}_i)_{i=0}^n$ be an extended partition. Let $(\mathcal{H}_i)_{i=0}^{n+1}$ be defined as in Definition~\ref{def:ext_part}. Then, there exists a unique probability measure $\mathbb{P}^{\otimes \Pi}$ on $(\Omega^{\mathcal{E}, \mathcal{X}}, \mathcal{F}^{\mathcal{E}, \mathcal{X}})$ such that
	\begin{enumerate}[label=(\roman*), topsep=0.3em, noitemsep]
		\item $\mathbb{P}^{\otimes \Pi}(A) = \mathbb{P}(A)$, for all $A \in \mathcal{H}_i$, $i = 0, ..., n+1$,
		
		\item $\mathbb{P}^{\otimes \Pi}(B \,|\, \mathcal{F}^{\mathcal{E}, \mathcal{X}}_{T_i}) = \mathbb{P}(B \,|\, \mathcal{G}_i)$, for all $B \in \mathcal{H}_{i+1}$, $i = 0, ..., n$.
	\end{enumerate}
\end{theorem}

\begin{proof}
	The proof is verbatim the same as that of \cite{MR3098443}, Theorem~4.3 (pp.\ 1598-1602, including all the related lemmas and cited results). We only need to replace their canonical space $\Omega^{\mathcal{E}, d}$ by ours. The key facts are that $\Omega^{\mathcal{E}, \mathcal{X}} = \mathcal{E} \times \mathcal{X}$ is a Polish space and $\mathcal{X}$ is $\Delta$-stable, which guarantee that the same proof works.
\end{proof}

\begin{lemma}\label{lem:ext_part}
	Let $\mathcal{E}_0$ be a Polish space, and set $\widetilde{\mathcal{E}} \coloneqq \mathcal{E}_0 \times \mathcal{E}$. On the augmented canonical space $\Omega^{\widetilde{\mathcal{E}}, \mathcal{X}} = \mathcal{E}_0 \times \mathcal{E} \times \mathcal{X}$, denote the projections by $(U, Z_0, X)$. Let $\mathbb{P}$ be a probability measure on $(\Omega^{\widetilde{\mathcal{E}}, \mathcal{X}}, \mathcal{F}^{\widetilde{\mathcal{E}}, \mathcal{X}})$. Let $Y$ be an $\mathbb{R}^d$-valued c\`adl\`ag adapted process with $Y_0 = 0$. Suppose that $\Delta(Y^T, S)$ is $\sigma(\Delta(X^T, S))$-measurable for all stopping times $0 \leq S \leq T$. Let $\Phi: \mathcal{E} \times D^d_0 \to D^{\mathcal{E}}$ be an updating function, and set $Z \coloneqq \Phi(Z_0, Y)$. Let $0 = T_0 \leq T_1 \leq \cdots \leq T_n < \infty$ be $\sigma(U)$-measurable random variables. In particular, each $T_i$ is a stopping time. Set $\mathcal{G}_i \coloneqq \sigma(U, Z_{T_i})$, $i = 0, ..., n$. Then,
	\begin{enumerate}[label=(\roman*), topsep=0.3em, noitemsep]
		\item $\Pi = (T_i, \mathcal{G}_i)_{i=0}^n$ is an extended partition,
		
		\item for each $t \geq 0$, the law of $Z_t$ under $\mathbb{P}$ agrees with the law of $Z_t$ under $\mathbb{P}^{\otimes \Pi}$.
	\end{enumerate}
\end{lemma}

\begin{proof}
	The proof of (i) (resp.\ (ii)) is exactly the same as in Step 2 (resp.\ Step 5) of the proof of \cite{MR3098443}, Theorem~7.1. In their proof, they have an explicit formula for each $T_i$ in terms of $U$, but the key point is the $\sigma(U)$-measurability of each $T_i$.
\end{proof}

We present several properties preserved by the concatenated measures. These results are analogous to those in \cite{MR3098443}, Section 4.2. The proofs follow almost verbatim from \cite{MR3098443}, so we omit them here. Again, the only difference is the choice of the canonical spaces. The continuity of sample paths assumed in \cite{MR3098443} is not essential for obtaining these results. What is crucial instead is the Polishness of $\Omega^{\mathcal{E}, \mathcal{X}}$ and the $\Delta$-stability of $\mathcal{X}$, which together ensure Theorem~\ref{thm:concat_meas}, and the subsequent properties follow from there.

\begin{proposition}[cf.\ \cite{MR3098443}, Proposition~4.10]\label{prop:sample_path}
	Let $\mathbb{P}$ be a probability measure on $(\Omega^{\mathcal{E}, \mathcal{X}}, \mathcal{F}^{\mathcal{E}, \mathcal{X}})$ and let $\Pi = (T_i, \mathcal{G}_i)_{i=0}^n$ be an extended partition. Let $A$ be an $\mathbb{R}^d$-valued continuous adapted process. Suppose that $\Delta(A, T_i)$ is $(\mathcal{G}_i \lor \sigma(\Delta(X, T_i)))$-measurable for all $i = 0, ..., n$. Then,
	\begin{enumerate}[label=(\roman*), topsep=0.3em, noitemsep]
		\item $A$ is $\mathbb{P}$-a.s.\ absolutely continuous if and only if $A$ is $\mathbb{P}^{\otimes \Pi}$-a.s.\ absolutely continuous,
		
		\item if $d$ is a square number and $A$ is square matrix-valued, $A_t - A_s$ is symmetric positive semi-definite for all $0 \leq s \leq t$ $\mathbb{P}$-a.s.\ if and only if $A_t - A_s$ is symmetric positive semi-definite for all $0 \leq s \leq t$ $\mathbb{P}^{\otimes \Pi}$-a.s.
	\end{enumerate}
\end{proposition}

\begin{proposition}[cf.\ \cite{MR3098443}, Proposition~4.11]\label{prop:exp_id}
	Let $\mathbb{P}$ be a probability measure on $(\Omega^{\mathcal{E}, \mathcal{X}}, \mathcal{F}^{\mathcal{E}, \mathcal{X}})$ and let $\Pi = (T_i, \mathcal{G}_i)_{i=0}^n$ be an extended partition. Let $A$ be an $\mathbb{R}^d$-valued continuous adapted process with $A_0 = 0$. Suppose that $\Delta(A, T_i)$ is $(\mathcal{G}_i \lor \sigma(\Delta(X, T_i)))$-measurable for all $i = 0, ..., n$. Moreover, suppose that $\alpha$ is an $\mathbb{R}^d$-valued progressively measurable process such that
	\begin{equation*}
		\begin{split}
			&\mathbb{P}\biggl(\int_0^t |\alpha_s| \,ds < \infty,\, A_t = \int_0^t \alpha_s \,ds,\, \forall\, t \geq 0\biggr)\\
			&\quad\quad= \mathbb{P}^{\otimes \Pi}\biggl(\int_0^t |\alpha_s| \,ds < \infty,\, A_t = \int_0^t \alpha_s \,ds,\, \forall\, t \geq 0\biggr)
			= 1.
		\end{split}
	\end{equation*}
	Then, for every nonnegative measurable function $f$ on $\mathbb{R}^d$, and every stopping time $T$ satisfying $(T - T_i)^+$ is $(\mathcal{G}_i \lor \sigma(\Delta(X, T_i)))$-measurable for all $i = 0, ..., n$, we have
	\begin{equation*}
		\mathbb{E}\biggl[\int_0^T f(\alpha_s) \,ds\biggr]
		= \mathbb{E}^{\otimes \Pi}\biggl[\int_0^T f(\alpha_s) \,ds\biggr].
	\end{equation*}
\end{proposition}

\begin{corollary}[cf.\ \cite{MR3098443}, Corollary~4.13]\label{cor:tight}
	Let $\mathbb{P}$ be a probability measure on $(\Omega^{\mathcal{E}, \mathcal{X}}, \mathcal{F}^{\mathcal{E}, \mathcal{X}})$, and for each $m \in \mathbb{N}^*$, let $\Pi^m = (T^m_i, \mathcal{G}^m_i)_{i=0}^{N(m)}$ be an extended partition. Let $A$ be an $\mathbb{R}^d$-valued continuous adapted process with $A_0 = 0$. Suppose that $T^m_i$ and $\Delta(A, T^m_i)$ are $(\mathcal{G}^m_i \lor \sigma(\Delta(X, T^m_i)))$-measurable for all $i = 0, ..., N(m)$ and $m \in \mathbb{N}^*$. Moreover, suppose that $\alpha$ is an $\mathbb{R}^d$-valued progressively measurable process such that
	\begin{equation*}
		\begin{split}
			&\mathbb{P}\biggl(\int_0^t |\alpha_s| \,ds < \infty,\, A_t = \int_0^t \alpha_s \,ds,\, \forall\, t \geq 0\biggr)\\
			&\quad\quad\quad\quad= \mathbb{P}^{\otimes \Pi^m}\biggl(\int_0^t |\alpha_s| \,ds < \infty,\, A_t = \int_0^t \alpha_s \,ds,\, \forall\, t \geq 0\biggr)
			= 1,\quad \forall\, m \in \mathbb{N}^*.
		\end{split}
	\end{equation*}
	Finally, suppose that
	\begin{equation*}
		\mathbb{E}\biggl[\int_0^t |\alpha_s| \,ds\biggr] < \infty,\quad
		\forall\, t > 0.
	\end{equation*}
	Then, the collection of probability measures $(\mathbb{P}^{\otimes \Pi^m} \circ A^{-1})_{m \in \mathbb{N}^*}$ on $C^d_0$ is tight.
\end{corollary}

\begin{lemma}[cf.\ \cite{MR3098443}, Theorem~4.15, Lemma~4.16]\label{lem:mtg_1}
	Let $\mathbb{P}$ be a probability measure on $(\Omega^{\mathcal{E}, \mathcal{X}}, \mathcal{F}^{\mathcal{E}, \mathcal{X}})$ and let $\Pi = (T_i, \mathcal{G}_i)_{i=0}^n$ be an extended partition. Let $M$ be a real-valued local martingale under $\mathbb{P}$ with bounded jumps, and $M_0 = 0$. Suppose that $\Delta(M, T_i)$ is $(\mathcal{G}_i \lor \sigma(\Delta(X, T_i)))$-measurable for all $i = 0, ..., n$. Then, $M$ is a local martingale under $\mathbb{P}^{\otimes \Pi}$.
\end{lemma}

\begin{lemma}[cf.\ \cite{MR3098443}, Theorem~4.15, Lemma~4.18]\label{lem:mtg_2}
	Let $\mathbb{P}$ be a probability measure on $(\Omega^{\mathcal{E}, \mathcal{X}}, \mathcal{F}^{\mathcal{E}, \mathcal{X}})$ and let $\Pi = (T_i, \mathcal{G}_i)_{i=0}^n$ be an extended partition. Let $M^1$, $M^2$ be real-valued local martingales under $\mathbb{P}$ with bounded jumps, and $M^1_0 = M^2_0 = 0$. Let $C$ be a real-valued continuous adapted process, with $C_0 = 0$, such that $M^3 \coloneqq M^1 M^2 - C$ is a local martingale under $\mathbb{P}$. Suppose that $\Delta(M^1, T_i)$, $\Delta(M^2, T_i)$, $\Delta(C, T_i)$ are $(\mathcal{G}_i \lor \sigma(\Delta(X, T_i)))$-measurable for all $i = 0, ..., n$. Then, $M^3$ is a local martingale under $\mathbb{P}^{\otimes \Pi}$.
\end{lemma}

\subsection{Approximation results}
The following result is about the weak convergence of the integrals of processes, which is analogous to that in \cite{MR3098443}, Section 6.1. Denote $\overline{\mathbb{N}}^* \coloneqq \mathbb{N}^* \cup \{\infty\}$.
\begin{proposition}[cf.\ \cite{MR3098443}, Proposition 6.1]\label{prop:wconv_int}
	Let $(Z^m)_{m \in \overline{\mathbb{N}}^*}$ be a collection of c\`adl\`ag $\mathcal{E}$-valued processes, possibly defined on different probability spaces with probability measures $(\mathbb{Q}^m)_{m \in \overline{\mathbb{N}}^*}$. Let $f: \mathbb{R}_+ \times \mathcal{E} \to \mathbb{R}^d$ be a measurable function. Suppose that
	\begin{enumerate}[label=(\roman*), topsep=0.3em, noitemsep]
		\item for each $t \geq 0$, the law of $Z^m_t$ under $\mathbb{Q}^m$ is independent of $m$ for $m \in \mathbb{N}^*$,
		
		\item the law of $Z^m$ on $D^{\mathcal{E}}$ under $\mathbb{Q}^m$ converges weakly to the law of $Z^\infty$ on $D^{\mathcal{E}}$ under $\mathbb{Q}^\infty$, as $m \to \infty$,
		
		\item $\mathbb{E}^{\mathbb{Q}^1}[\int_0^t |f(s, Z^1_s)| \,ds] < \infty$, $\forall\, t > 0$.
	\end{enumerate}
	Then, for each $m \in \overline{\mathbb{N}}^*$, the process $F^m_t \coloneqq \int_0^t f(s, Z^m_s) \,ds$ is well-defined and absolutely continuous $\mathbb{Q}^m$-a.s. Moreover, the following hold:
	\begin{enumerate}[topsep=0.3em, noitemsep]
		\item[(iv)] the collection $(f(\cdot, Z^m_\cdot), \mathrm{Leb}([0, t]) \otimes \mathbb{Q}^m)_{m \in \overline{\mathbb{N}}^*}$ is uniformly integrable, for every $t > 0$,
		
		\item[(v)] the law of $(Z^m, F^m)$ on $D^{\mathcal{E}} \times C^d_0$ under $\mathbb{Q}^m$ converges weakly to the law of $(Z^\infty, F^\infty)$ on $D^{\mathcal{E}} \times C^d_0$ under $\mathbb{Q}^\infty$, as $m \to \infty$.
	\end{enumerate}
\end{proposition}

\begin{proof}
	The proof is almost the same as in the proof of \cite{MR3098443}, Proposition 6.1, with $C^{\mathcal{E}}$ replaced by $D^{\mathcal{E}}$. Only two places need slight changes. First, since $Z^\infty$ is c\`adl\`ag, (ii) implies that $Z^m_t$ converges to $Z^\infty_t$ in law for all but countably many $t \geq 0$. Together with (i) and the right-continuity of the sample paths of $Z^\infty$, we still obtain that $Z^\infty_t$ and $Z^1_t$ have the same law for all $t \geq 0$. Secondly, in the proof of (v), one needs to check that the map $D^{\mathcal{E}} \ni z \mapsto \int_0^{\cdot \land k} f^k(s, z(s)) \,ds \in C^d_0$ is continuous, where $f^k: [0, k] \times \mathcal{E} \to \mathbb{R}^d$ is a bounded continuous function. Since convergence in the Skorokhod space implies pointwise convergence almost everywhere, the dominated convergence theorem then finishes the proof.
\end{proof}

The following technical lemma characterizes under the concatenated probability measures, how ``close'' the integral of a process is to the integral of its certain type of conditional expectation. This result will be used in the proof of Theorem~\ref{thm:mp} to find the characteristics of some limiting process.

\begin{lemma}\label{lem:approx}
	Set $\widetilde{\mathcal{E}} \coloneqq [0, 1] \times \mathcal{E}$. On the augmented canonical space $\Omega^{\widetilde{\mathcal{E}}, \mathcal{X}} = [0, 1] \times \mathcal{E} \times \mathcal{X}$, denote the projections by $(U, Z_0, X)$. Let $\mathbb{P}$ be a probability measure on $(\Omega^{\widetilde{\mathcal{E}}, \mathcal{X}}, \mathcal{F}^{\widetilde{\mathcal{E}}, \mathcal{X}})$ under which $U \sim \mathrm{Unif}([0, 1])$ is independent of $(Z_0, X)$. Let $Y$ be an $\mathbb{R}^d$-valued c\`adl\`ag adapted process with $Y_0 = 0$, and $A$ be an $\mathbb{R}^n$-valued continuous adapted process with $A_0 = 0$. Suppose that $\Delta(Y^T, S)$, $\Delta(A^T, S)$ are $\sigma(\Delta(X^T, S))$-measurable for all stopping times $0 \leq S \leq T$. Let $\Phi: \mathcal{E} \times D^d_0 \to D^{\mathcal{E}}$ be an updating function, and set $Z \coloneqq \Phi(Z_0, Y)$. For $m \in \mathbb{N}^*$, set $N(m) \coloneqq m^2$. Define the stopping times $T^m_0 \coloneqq 0$, $T^m_i \coloneqq (U + i-1) / m$, $i = 1, ..., N(m)$. Set $\mathcal{G}^m_i \coloneqq \sigma(U, Z_{T^m_i})$, $i = 0, ..., N(m)$, and $\Pi^m \coloneqq (T^m_i, \mathcal{G}^m_i)_{i=0}^{N(m)}$. Let $\alpha$ be an $\mathbb{R}^n$-valued progressively measurable process. Suppose that $\mathbb{E}[\int_0^t |\alpha_s| \,ds] < \infty$,
	$\forall\, t > 0$, and
	\begin{equation*}
		\begin{split}
			&\mathbb{P}\biggl(\int_0^t |\alpha_s| \,ds < \infty,\, A_t = \int_0^t \alpha_s \,ds,\, \forall\, t \geq 0\biggr)\\
			&\quad\quad\quad\quad= \mathbb{P}^{\otimes \Pi^m}\biggl(\int_0^t |\alpha_s| \,ds < \infty,\, A_t = \int_0^t \alpha_s \,ds,\, \forall\, t \geq 0\biggr)
			= 1,\quad \forall\, m \in \mathbb{N}^*.
		\end{split}
	\end{equation*}
	Let $\widehat{a}: \mathbb{R}_+ \times \mathcal{E} \to \mathbb{R}^n$ be a measurable function such that $\widehat{a}(t, Z_t) = \mathbb{E}[\alpha_t \,|\, Z_t]$ for Lebesgue-a.e.\ $t \geq 0$. Set $\overline{A} \coloneqq \int_0^\cdot \widehat{a}(s, Z_s) \,ds$. Then, for any $\varepsilon > 0$ and $t > 0$,
	\begin{equation*}
		\lim_{m \to \infty} \mathbb{P}^{\otimes \Pi^m}\biggl(\max_{s \leq t} |A_s - \overline{A}_s| \geq \varepsilon\biggr) = 0.
	\end{equation*}
\end{lemma}

\begin{proof}
	By Lemma~\ref{lem:ext_part}, $\Pi^m$ is indeed an extended partition, and the law of $Z_t$ under $\mathbb{P}$ agrees with the law of $Z_t$ under $\mathbb{P}^{\otimes \Pi^m}$ for each $t \geq 0$ and $m \in \mathbb{N}^*$. By the definition of $\widehat{a}$ and Jensen's inequality, we have $\mathbb{E}[\int_0^t |\widehat{a}(s, Z_s)| \,ds] < \infty$,
	$\forall\, t > 0$, thus $\overline{A}$ is well-defined under $\mathbb{P}$ and all $\mathbb{P}^{\otimes \Pi^m}$. This implies that the collection $(\widehat{a}(\cdot, Z_\cdot), \mathrm{Leb}([0, t]) \otimes \mathbb{P}^{\otimes \Pi^m})_{m \in \mathbb{N}^*}$ is uniformly integrable, for every $t > 0$. The rest of the proof follows exactly the same as in Step 7 of the proof of \cite{MR3098443}, Theorem~7.1.
\end{proof}

\subsection{Other lemmas and notation}
The first lemma is measure theoretic. Recall that if $f: \mathbb{R}_+ \to \mathbb{R}$ is a right-continuous finite variation function with $f(0) = 0$, then it induces a measure $\lambda$ on $\mathbb{R}_+$ that satisfies $\lambda([0, t]) = f(t)$ for all $t \geq 0$. Now suppose that $\mu$ is a function on $\mathbb{R}_+$ taking values of Borel measures on $\mathbb{R}^d$ with $\mu_0 = 0$. We are interested in finding a measure $\nu$ on $\mathbb{R}_+ \times \mathbb{R}^d$ such that $\nu([0, t] \times A) = \mu_t(A)$ for all $t \geq 0$ and $A \in \mathcal{B}(\mathbb{R}^d)$. The following lemma provides a sufficient condition, which serves our purpose to prove the main theorem.

\begin{lemma}\label{lem:cstr_nu}
	Let $\mu \in C^{\mathcal{M}_+, d}_{0, \mathrm{i}}$. Then, there exists a $\sigma$-finite positive Borel measure $\nu$ on $\mathbb{R}_+ \times \mathbb{R}^d$ such that
	\begin{equation}\label{eq:nu_mu}
		\nu([0, t] \times A) = \mu_t(A),\quad
		\forall\, t \geq 0,\, A \in \mathcal{B}(\mathbb{R}^d).
	\end{equation}
\end{lemma}

\begin{proof}
	First we notice the following fact. Since $t \mapsto \mu_t$ is continuous in the sense of weak convergence, we know that $t \mapsto \mu_t(\mathbb{R}^d)$ is a continuous function. For $0 \leq s < t$, by assumption $\mu_t - \mu_s$ is a positive measure, so we have
	\begin{equation}\label{eq:mu_cts}
		0 \leq \mu_t(A) - \mu_s(A)
		\leq \mu_t(\mathbb{R}^d) - \mu_s(\mathbb{R}^d),\quad
		A \in \mathcal{B}(\mathbb{R}^d).
	\end{equation}
	This implies that $t \mapsto \mu_t(A)$ is continuous for all $A \in \mathcal{B}(\mathbb{R}^d)$.
	
	For $0 \leq s \leq t \leq \infty$ and $A \in \mathcal{B}(\mathbb{R}^d)$, we define a set function
	\begin{equation}\label{eq:nu_mu_2}
		\nu([s, t) \times A)
		\coloneqq \mu_t(A) - \mu_s(A),
	\end{equation}
	where we use the convention $\mu_\infty(A) \coloneqq \lim_{t \to \infty} \mu_t(A)$, which is well-defined by monotonicity (but can be infinite). We also define a collection of subsets of $\mathbb{R}_+ \times \mathbb{R}^d$ via
	\begin{equation*}
		\mathcal{A}
		\coloneqq \Biggl\{\bigcup_{i = 1}^n ([s_i, t_i) \times A_i): 0 \leq s_i \leq t_i \leq \infty,\, A_i \in \mathcal{B}(\mathbb{R}^d),\, 1 \leq i \leq n,\, n \in \mathbb{N}\Biggr\},
	\end{equation*}
	which is an algebra that generates $\mathcal{B}(\mathbb{R}_+ \times \mathbb{R}^d)$. For each $E \in \mathcal{A}$, we can write it in the form $E = \bigcup_{i = 1}^n ([s_i, t_i) \times A_i)$, where the sets $[s_i, t_i) \times A_i$ are disjoint. Then, we define
	\begin{equation*}
		\nu(E) \coloneqq \sum_{i=1}^n \nu([s_i, t_i) \times A_i).
	\end{equation*}
	It is straightforward to verify that $\nu(E)$ is well-defined, i.e.\ it does not depend on how $E$ is partitioned. So far we have defined a set function $\nu$ on $\mathcal{A}$, which satisfies $\nu(\varnothing) = 0$ and is finitely additive. If we manage to show $\nu$ is $\sigma$-additive on $\mathcal{A}$, then by Carath\'eodory's extension theorem, we could uniquely extend $\nu$ to a measure on $\mathcal{B}(\mathbb{R}_+ \times \mathbb{R}^d)$, and \eqref{eq:nu_mu} follows from \eqref{eq:nu_mu_2} and the continuity of $\mu$. This would finish the proof.
	
	It only remains to prove $\nu$ is $\sigma$-additive on $\mathcal{A}$. To prove this, it suffices to show the following statement: \emph{if $[s, t) \times A = \bigcup_{i=1}^\infty ([s_i, t_i) \times A_i)$, where $0 \leq s < t \leq \infty, 0 \leq s_i < t_i \leq \infty, A, A_i \in \mathcal{B}(\mathbb{R}^d)$, and the sets $[s_i, t_i) \times A_i$ are disjoint, then}
	\begin{equation*}
		\nu([s, t) \times A)
		= \sum_{i=1}^\infty \nu([s_i, t_i) \times A_i).
	\end{equation*}
	One direction of inequality is simple. For each $n \in \mathbb{N}^*$, we have $[s, t) \times A \supseteq \bigcup_{i=1}^n ([s_i, t_i) \times A_i)$. By the finiteness, it is easy to check that $\nu([s, t) \times A) \geq \sum_{i=1}^n \nu([s_i, t_i) \times A_i)$. Sending $n \to \infty$ proves the ``$\geq$'' direction. Conversely, let us first assume that $s > 0$ and $t < \infty$. Pick any $\varepsilon > 0$. Using \eqref{eq:nu_mu_2} and the continuity of $t \mapsto \mu_t(A)$, one can find $t^\prime \in (s, t)$ such that $\nu([s, t^\prime) \times A) > \nu([s, t) \times A) - \varepsilon/4$. Next, since $\mu_{t^\prime} - \mu_s$ is a finite positive Borel measure on $\mathbb{R}^d$, by the regularity one can find a compact set $K \subseteq A$ such that $\nu([s, t^\prime) \times K) > \nu([s, t^\prime) \times A) - \varepsilon/4$. Combining these two steps gives us
	\begin{equation}\label{eq:reg}
		\nu([s, t^\prime) \times K)
		> \nu([s, t) \times A) - \frac{\varepsilon}{2}.
	\end{equation}
	Similarly, for each $i \in \mathbb{N}^*$, one can find $s_i^\prime \in (0, s_i)$ and an open set $U_i \supseteq A_i$ such that
	\begin{equation}\label{eq:reg_2}
		\nu([s_i^\prime, t_i) \times U_i)
		< \nu([s_i, t_i) \times A_i) + \frac{\varepsilon}{2^{i+1}}.
	\end{equation}
	Note that $[s, t^\prime] \times K$ is a compact set, and we have $[s, t^\prime] \times K \subseteq \bigcup_{i=1}^\infty ((s_i^\prime, t_i) \times U_i)$. Thus, we can extract a finite subcover $[s, t^\prime] \times K \subseteq \bigcup_{i=1}^n ((s_i^\prime, t_i) \times U_i)$, which leads to $[s, t^\prime) \times K \subseteq \bigcup_{i=1}^n ([s_i^\prime, t_i) \times U_i)$. By the finiteness and \eqref{eq:reg}, \eqref{eq:reg_2}, we obtain the estimate
	\begin{equation*}
		\begin{split}
			\nu([s, t) \times A)
			&< \nu([s, t^\prime) \times K) + \frac{\varepsilon}{2}
			\leq \sum_{i=1}^n \nu([s_i^\prime, t_i) \times U_i) + \frac{\varepsilon}{2}
			\leq \sum_{i=1}^\infty \nu([s_i^\prime, t_i) \times U_i) + \frac{\varepsilon}{2}\\
			&\leq \sum_{i=1}^\infty \biggl(\nu([s_i, t_i) \times A_i) + \frac{\varepsilon}{2^{i+1}}\biggr) + \frac{\varepsilon}{2}
			= \sum_{i=1}^\infty \nu([s_i, t_i) \times A_i) + \varepsilon.
		\end{split}
	\end{equation*}
	Sending $\varepsilon \to 0$ finishes the proof for $s > 0$ and $t < \infty$. Finally, when $s = 0$, for those $i$ with $s_i = 0$, the interval $[0, t_i)$ is relatively open in $\mathbb{R}_+$, so we may take $s_i^\prime = s_i = 0$, and the interval $[0, t_i)$ serves our purpose. When $t = \infty$, we simply partition $[s, \infty)$ into countably many subintervals $[s, t^j)$ with all $t^j < \infty$. We then apply what we just proved to each $[s, t^j) \times A$ and sum up the results.
\end{proof}

\begin{remark}
	The condition $\mu \in C^{\mathcal{M}_+, d}_{0, \mathrm{i}}$ is far from optimal but sufficient for our usage. Analogous to real-valued functions, it is reasonable to expect that the conclusion of Lemma~\ref{lem:cstr_nu} remains valid for functions $\mu$ that are of ``finite variation'' in a suitable sense. The condition $\mu \in C^{\mathcal{M}_+, d}_{0, \mathrm{i}}$ says that $\mu$ is continuous and nondecreasing, thus $\mu$ is expected to be in this ``finite variation'' class.
\end{remark}

The next lemma reflects the fact that, in the locally bounded case, the local martingale property on a general filtered probability space is preserved when passing to the canonical space with the natural filtration.

\begin{lemma}\label{lem:loc_mtg}
	Let $\mathcal{E}$ be a Polish space. Let $(\Omega, \mathcal{F}, (\mathcal{F}_t)_{t \geq 0}, \mathbb{P})$ be a filtered probability space that supports an $\mathcal{E}$-valued c\`adl\`ag adapted process $X$. Let $\Omega^*$ be a closed subset of $D^{\mathcal{E}}$ with $\mathbb{P}(X \in \Omega^*) = 1$. Let $X^*$ be the canonical process on $\Omega^*$, $\mathcal{F}^* = \sigma(X^*)$, and $(\mathcal{F}^*_t)_{t \geq 0}$ be the natural filtration of $X^*$. Let $\mathbb{P}^*$ be the law of $X$ on $\Omega^*$ under $\mathbb{P}$. Let $\Psi: D^{\mathcal{E}} \to D^d_0$ be a measurable map satisfying
	\begin{enumerate}[label=(\roman*), topsep=0.3em, noitemsep]
		\item nonanticipativity:
		\begin{equation*}
			\Psi(x)^t = \Psi(x^t)^t,\quad
			\forall\, t \geq 0,\, x \in D^{\mathcal{E}},
		\end{equation*}
		
		\item bounded jumps: there exists $M > 0$ such that
		\begin{equation*}
			|\Psi(x)(t) - \Psi(x)(t-)| \leq M,\quad
			\forall\, t \geq 0,\, x \in D^{\mathcal{E}}.
		\end{equation*}
	\end{enumerate}
	Let $F: \mathbb{R}^d \to \mathbb{R}^{d^\prime}$ be a continuous function. Suppose that the process $t \mapsto F(\Psi(X)_t)$ is a local martingale on $(\Omega, \mathcal{F}, (\mathcal{F}_t)_{t \geq 0}, \mathbb{P})$. Then, the process $t \mapsto F(\Psi(X^*)_t)$ is a local martingale on $(\Omega^*, \mathcal{F}^*, (\mathcal{F}^*_t)_{t \geq 0}, \mathbb{P}^*)$.
\end{lemma}

\begin{proof}
	For $n \in \mathbb{N}^*$, define $\mathcal{T}_n: D^d_0 \to [0, \infty]$ via $\mathcal{T}_n(y) \coloneqq \inf\{t > 0: |y(t)| \geq n\}$, $y \in D^d_0$, then define $\Phi_n: D^d_0 \to D^d_0$ via $\Phi_n(y) \coloneqq y^{\mathcal{T}_n(y)}$, $y \in D^d_0$. Note that both $\mathcal{T}_n$ and $\Phi_n$ are measurable maps. Assumption (i) tells us that $\Psi(X)$ is adapted to $(\mathcal{F}_t)_{t \geq 0}$, and $\Psi(X^*)$ is adapted to $(\mathcal{F}^*_t)_{t \geq 0}$. Thus, $\tau_n \coloneqq \mathcal{T}_n(\Psi(X))$ is an $(\mathcal{F}_t)_{t \geq 0}$-stopping time, and $\tau^*_n \coloneqq \mathcal{T}_n(\Psi(X^*))$ is an $(\mathcal{F}^*_t)_{t \geq 0}$-stopping time. By assumption (ii), we have that $|\Phi_n \circ \Psi(x)| \leq n + M$ for all $x \in D^{\mathcal{E}}$, so both $\Psi(X)^{\tau_n} = \Phi_n \circ \Psi(X)$ and $\Psi(X^*)^{\tau^*_n} = \Phi_n \circ \Psi(X^*)$ are bounded processes. It is easy to check that $\Phi_n \circ \Psi$ satisfies (i) and (ii). Therefore, it suffices to prove the lemma for the case where $\Psi(x)(t)$ is uniformly bounded in $x \in D^{\mathcal{E}}$ and $t \geq 0$. In particular, by the continuity of $F$, in this case both processes $t \mapsto F(\Psi(X)_t)$ and $t \mapsto F(\Psi(X^*)_t)$ are bounded.
	
	Now assume that $t \mapsto F(\Psi(X)_t)$ is a bounded martingale on $(\Omega, \mathcal{F}, (\mathcal{F}_t)_{t \geq 0}, \mathbb{P})$, and we prove $t \mapsto F(\Psi(X^*)_t)$ is a bounded martingale on $(\Omega^*, \mathcal{F}^*, (\mathcal{F}^*_t)_{t \geq 0}, \mathbb{P}^*)$. Let $0 \leq s < t$. Our goal is to show
	\begin{equation*}
		\mathbb{E}^*[F(\Psi(X^*)_t) \bm{1}_F]
		= \mathbb{E}^*[F(\Psi(X^*)_s) \bm{1}_F],\quad
		\forall\, F \in \mathcal{F}^*_s.
	\end{equation*}
	By Dynkin's $\pi$-$\lambda$ theorem, it suffices to take $F$ of the form $\{X^*_{s_1} \in A_1, ..., X^*_{s_n} \in A_n\}$, where $0 \leq s_1 < \cdots < s_n \leq s$, $A_1, ..., A_n \in \mathcal{B}(\mathcal{E})$ and $n \in \mathbb{N}^*$. Then, by the definition of $\mathbb{P^*}$ and the martingale property of $\Psi(X)$ on $(\Omega, \mathcal{F}, (\mathcal{F}_t)_{t \geq 0}, \mathbb{P})$, it follows that
	\begin{equation*}
		\begin{split}
			\mathbb{E}^*\bigl[F(\Psi(X^*)_t) \bm{1}_{\{X^*_{s_1} \in A_1, ..., X^*_{s_n} \in A_n\}}\bigr]
			&= \mathbb{E}\bigl[F(\Psi(X)_t) \bm{1}_{\{X_{s_1} \in A_1, ..., X_{s_n} \in A_n\}}\bigr]\\
			&= \mathbb{E}\bigl[F(\Psi(X)_s) \bm{1}_{\{X_{s_1} \in A_1, ..., X_{s_n} \in A_n\}}\bigr]\\
			&= \mathbb{E}^*\bigl[F(\Psi(X^*)_s) \bm{1}_{\{X^*_{s_1} \in A_1, ..., X^*_{s_n} \in A_n\}}\bigr],
		\end{split}
	\end{equation*}
	which finishes the proof.
\end{proof}

The next lemma deals with the joint convergence in law of a coupling of two convergent sequences.

\begin{lemma}\label{lem:joint_wconv}
	Let $\mathcal{E}_1$, $\mathcal{E}_2$, $\mathcal{E}_3$ be Polish spaces. Let $(Y^m, Z^m)_{m \in \overline{\mathbb{N}}^*}$ be a collection of $(\mathcal{E}_1 \times \mathcal{E}_2)$-valued random variables, possibly defined on different probability spaces with probability measures $(\mathbb{P}^m)_{m \in \overline{\mathbb{N}}^*}$. Let $f: \mathcal{E}_2 \to \mathcal{E}_3$ be a measurable function. Suppose that $\mathbb{P}^m \circ (Y^m, Z^m)^{-1} \Rightarrow \mathbb{P}^\infty \circ (Y^\infty, Z^\infty)^{-1}$ and $\mathbb{P}^m \circ (Z^m, f(Z^m))^{-1} \Rightarrow \mathbb{P}^\infty \circ (Z^\infty, f(Z^\infty))^{-1}$, as $m \to \infty$. Then, $\mathbb{P}^m \circ (Y^m, Z^m, f(Z^m))^{-1} \Rightarrow \mathbb{P}^\infty \circ (Y^\infty, Z^\infty, f(Z^\infty))^{-1}$, as $m \to \infty$.
\end{lemma}

\begin{proof}
	First we take any subsequence $(m_k)_{k \in \mathbb{N^*}}$ of $\mathbb{N}^*$. Consider the sequence of probability measures $(\mathbb{P}^{m_k} \circ (Y^{m_k}, Z^{m_k}, Z^{m_k}, f(Z^{m_k}))^{-1})_{k \in \mathbb{N}^*}$ on $\mathcal{E}_1 \times \mathcal{E}_2 \times \mathcal{E}_2 \times \mathcal{E}_3$. This is a tight sequence, since it is a coupling of two tight sequences of probability measures on $\mathcal{E}_1 \times \mathcal{E}_2$ and $\mathcal{E}_2 \times \mathcal{E}_3$. Then, there exists a further subsequence $(m_{k_l})_{l \in \mathbb{N^*}}$ of $(m_k)_{k \in \mathbb{N^*}}$ and a probability measure $\mu$ on $\mathcal{E}_1 \times \mathcal{E}_2 \times \mathcal{E}_2 \times \mathcal{E}_3$, such that $\mathbb{P}^{m_{k_l}} \circ (Y^{m_{k_l}}, Z^{m_{k_l}}, Z^{m_{k_l}}, f(Z^{m_{k_l}}))^{-1} \Rightarrow \mu$. Let $\mu_{12}(dx_1, dx_2)$, $\mu_{34}(dx_3, dx_4)$ denote the marginals of $\mu$ on the first and last two coordinates, and let $\mu(dx_1, dx_2, dx_3, dx_4) = \mu_{34|12}(x_1, x_2, dx_3, dx_4) \mu_{12}(dx_1, dx_2)$ be the disintegration. We know that $\mu_{12} = \mathbb{P}^\infty \circ (Y^\infty, Z^\infty)^{-1}$ and $\mu_{34} = \mathbb{P}^\infty \circ (Z^\infty, f(Z^\infty))^{-1}$. We also know from the Portmanteau theorem that
	\begin{equation*}
		\mu(\{x_2 = x_3\})
		\geq \limsup_{l \to \infty} \mathbb{P}^{m_{k_l}}(Z^{m_{k_l}} = Z^{m_{k_l}}) = 1.
	\end{equation*}
	This implies that $\mu_{34|12}(x_1, x_2, dx_3, dx_4) = \delta_{(x_2, f(x_2))}(dx_3, dx_4)$, thus we conclude that $\mu = \mathbb{P}^\infty \circ (Y^\infty, Z^\infty, Z^\infty, f(Z^\infty))^{-1}$.
	
	So far we have proved that for any subsequence of $(\mathbb{P}^m \circ (Y^m, Z^m, f(Z^m))^{-1})$, there exists a further subsequence that converges weakly to $\mathbb{P}^\infty \circ (Y^\infty, Z^\infty, f(Z^\infty))^{-1}$. This implies the weak convergence of the whole sequence, and the proof is complete.
\end{proof}

\begin{remark}
	If $f$ is a continuous function, then by the continuous mapping theorem, $\mathbb{P}^m \circ (Y^m, Z^m)^{-1} \Rightarrow \mathbb{P}^\infty \circ (Y^\infty, Z^\infty)^{-1}$ implies $\mathbb{P}^m \circ (Y^m, Z^m, f(Z^m))^{-1} \Rightarrow \mathbb{P}^\infty \circ (Y^\infty, Z^\infty, f(Z^\infty))^{-1}$. However, in Lemma~\ref{lem:joint_wconv} we only assume the measurability of $f$. Thus, the extra assumption $\mathbb{P}^m \circ (Z^m, f(Z^m))^{-1} \Rightarrow \mathbb{P}^\infty \circ (Z^\infty, f(Z^\infty))^{-1}$ is needed.
\end{remark}

We introduce below a short notation for the running integral of an optional random function against a random measure, following \cite{MR1943877}, Equation II.1.5.

\begin{definition}
	Let $(\Omega, \mathcal{F}, (\mathcal{F}_t)_{t \geq 0}, \mathbb{P})$ be a filtered probability space. Let $\mathcal{O}$ be the optional $\sigma$-algebra on $\Omega \times \mathbb{R}_+$, and $\mu: \Omega \times \mathcal{B}(\mathbb{R}_+ \times \mathbb{R}^d) \to [0, \infty]$ be a random measure. Let $W: \Omega \times \mathbb{R}_+ \times \mathbb{R}^d \to \mathbb{R}$ be a measurable function with respect to $\mathcal{O} \otimes \mathcal{B}(\mathbb{R}^d)$. We define the process $W * \mu$ via
	\begin{equation*}
		(W * \mu)_t(\omega) \coloneqq
		\int_{[0, t] \times \mathbb{R}^d} W(\omega, s, x) \,\mu(\omega, ds, dx),
	\end{equation*}
	whenever $W(\omega, \cdot)$ is integrable with respect to $\mu(\omega, \cdot)$. Otherwise, set $(W * \mu)_t(\omega) = \infty$. If $W$ is $\mathbb{R}^d$-valued, we define $W * \mu$ component-wise.
\end{definition}

The next definition is about a property called \emph{convergence determining}. This notion is already introduced in \cite{MR1943877}, Definition~VII.2.7. We rephrase it below, and stick to their notation.

\begin{definition}\label{def:conv_det}
	Let $C_1(\mathbb{R}^d)$ be any subclass of nonnegative bounded continuous functions from $\mathbb{R}^d$ to $\mathbb{R}$ which are $0$ in a neighborhood of $0$, containing all functions of the form $(a|x| - 1)^+ \land 1$, $a \in \mathbb{Q}_+$, and satisfying the following property: let $(\eta_n)$, $\eta$ be positive Borel measures on $\mathbb{R}^d$ which do not charge $\{0\}$ and are finite on $\{x: |x| \geq r\}$ for all $r > 0$, then $\eta_n(f) \to \eta(f)$ for all $f \in C_1(\mathbb{R}^d)$ implies $\eta_n(f) \to \eta(f)$ for all bounded continuous functions $f$ which are $0$ in a neighborhood of $0$. We call $C_1(\mathbb{R}^d)$ a convergence determining class (for the weak convergence induced by bounded continuous functions which are $0$ in a neighborhood of $0$).
\end{definition}

\begin{remark}\label{rem:conv_det}
	As was mentioned in \cite{MR1943877}, right after Definition~VII.2.7, there exists a class $C_1(\mathbb{R}^d)$ which is countable. This will be convenient when proving our main theorem. Note that convergence determining implies \emph{measure determining}: let $\eta$, $\eta^\prime$ be positive Borel measures on $\mathbb{R}^d$ which do not charge $\{0\}$ and are finite on $\{x: |x| \geq r\}$ for all $r > 0$, then $\eta(f) = \eta^\prime(f)$ for all $f \in C_1(\mathbb{R}^d)$ implies $\eta = \eta^\prime$.
\end{remark}

\section{Proof of Theorem~\ref{thm:mp}}\label{sec:4}
With all the preparations in Section~\ref{sec:3}, we are now able to prove our main results. To better align with the proof of \cite{MR3098443}, Theorem~7.1, we will break our proof into several steps.

\begin{proof}[Proof of Theorem~\ref{thm:mp}]
	The existence of $\widehat{b}$, $\widehat{c}$ and $\widehat{\kappa}$ satisfying \eqref{eq:cond_exp} follows from \cite{MR3098443}, Proposition~5.1 and \cite{MR4814246}, Lemma~2.5 (which obviously extends to $\mathcal{E}$-valued processes and transition kernels from $\mathbb{R}_+ \times \mathcal{E}$ to $\mathbb{R}^d$). Also, without loss of generality, we may assume that $h$ is continuous. Otherwise, we can take a continuous truncation function $\widetilde{h}$ and prove the theorem. With back-and-forth applications of \eqref{eq:B(h)}, we first compute the characteristics of $Y$ associated with $\widetilde{h}$, apply the theorem with $\widetilde{h}$, then compute the characteristics of $\widehat{Y}$ associated with $h$. This argument is similar to the proof of Corollary~\ref{cor:mp}.
	
	\medskip
	\emph{Step 1: Canonical space and processes.}\quad We define the measure-valued process $M$ on $\Omega$ via
	\begin{equation*}
		M_t(A)
		\coloneqq \int_{[0, t] \times A} 1 \land |\xi|^2 \,\nu(ds, d\xi)
		= \int_0^t \int_A 1 \land |\xi|^2 \,\kappa_s(d\xi) \,ds,\quad
		t \geq 0,\, A \in \mathcal{B}(\mathbb{R}^d).
	\end{equation*}
	Then, the random object $(Z_0, Y, B, C, M)$ takes values in $\mathcal{E} \times D^d_0 \times C^d_0 \times C^{d^2}_0 \times C^{\mathcal{M}_+, d}_{0, \mathrm{i}}$ $\mathbb{P}$-a.s. In order to utilize the approximation results developed in \cite{MR3098443}, we need to use a randomized discretization of time, which leads to an extra dimension. Thus, we define our canonical space as $\Omega^* \coloneqq [0, 1] \times \mathcal{E} \times D^d_0 \times C^d_0 \times C^{d^2}_0 \times C^{\mathcal{M}_+, d}_{0, \mathrm{i}}$. By viewing $\mathcal{E}^* \coloneqq [0, 1] \times \mathcal{E}$ as a new Polish space, and noticing $\mathcal{X}^* \coloneqq D^d_0 \times C^d_0 \times C^{d^2}_0 \times C^{\mathcal{M}_+, d}_{0, \mathrm{i}}$ is a $\Delta$-stable closed subset of $D^{\mathcal{E}^{*\prime}}_0$ with $\mathcal{E}^{*\prime} = \mathbb{R}^d \times \mathbb{R}^d \times \mathbb{R}^{d^2} \times \mathcal{M}_+(\mathbb{R}^d)$, one can write $\Omega^* = \Omega^{\mathcal{E}^*, \mathcal{X}^*} = \mathcal{E}^* \times \mathcal{X}^*$, so all the results established in Section~\ref{sec:3} apply to $\Omega^*$. The generic element of $\Omega^*$ is denoted by $\omega = (u, \varepsilon, \eta, \beta, \gamma, \mu)$, and the projections are denoted by
	\begin{equation*}
		U^*(\omega)   \coloneqq u,\quad
		Z_0^*(\omega) \coloneqq \varepsilon,\quad
		Y^*(\omega)   \coloneqq \eta,\quad
		B^*(\omega)   \coloneqq \beta,\quad
		C^*(\omega)   \coloneqq \gamma,\quad
		M^*(\omega)   \coloneqq \mu.
	\end{equation*}
	We also write $x = (\eta, \beta, \gamma, \mu)$ and $X = (Y^*, B^*, C^*, M^*)$. Let $\mathcal{F}^* \coloneqq \sigma(U^*, Z_0^*, X)$ and $\mathcal{F}^*_t \coloneqq \sigma(U^*, Z_0^*, X^t)$ for $t \geq 0$. Denote $\mathbb{F}^* = (\mathcal{F}^*_t)_{t \geq 0}$ and let $\widetilde{\mathbb{F}}^* = (\widetilde{\mathcal{F}}^*_t)_{t \geq 0}$ be the right-continuous regularization of $\mathbb{F}^*$, i.e.\ $\widetilde{\mathcal{F}}^*_t \coloneqq \bigcap_{s > t} \mathcal{F}^*_s$, $t \geq 0$.\footnote{To apply the theory of characteristics of semimartingales established in \cite{MR1943877}, we need to work with right-continuous filtrations. This is only for technical reasons, and barely complicates our proof.} Unless otherwise stated, we always refer to the natural filtration $\mathbb{F}^*$. When we work with characteristics, we will explicitly mention $\widetilde{\mathbb{F}}^*$. We define a probability measure $\mathbb{Q}$ on $\Omega^*$, which is the product of the Lebesgue measure on $[0, 1]$ and the law of $(Z_0, Y, B, C, M)$ on $\mathcal{E} \times D^d_0 \times C^d_0 \times C^{d^2}_0 \times C^{\mathcal{M}_+, d}_{0, \mathrm{i}}$ under $\mathbb{P}$, i.e.\
	\begin{equation*}
		\mathbb{Q} \coloneqq \mathrm{Leb}([0, 1]) \otimes (\mathbb{P} \circ (Z_0, Y, B, C, M)^{-1}).
	\end{equation*}
	Then under $\mathbb{Q}$, we know that $U^* \sim \mathrm{Unif}([0, 1])$, $(Z_0^*, Y^*, B^*, C^*, M^*)$ has the same joint law as $(Z_0, Y, B, C, M)$, and it is independent of $U^*$.
	
	According to Lemma~\ref{lem:cstr_nu}, for each $\omega \in \Omega^*$, there exists a Borel measure $\lambda^*(\omega; \cdot)$ on $\mathbb{R}_+ \times \mathbb{R}^d$ such that\footnote{For a random measure like $M^*_t(\omega; \cdot)$, we often omit its dependency on $\omega$ and simply write $M^*_t(\cdot)$.}
	\begin{equation*}
		\lambda^*(\omega; [0, t] \times A) = M^*_t(\omega; A),\quad
		\forall\, t \geq 0,\, A \in \mathcal{B}(\mathbb{R}^d).
	\end{equation*}
	By Dynkin's $\pi$-$\lambda$ theorem, it is easy to see that $\lambda^*$ is a random measure, i.e.\ $\lambda^*(\omega; E)$ is measurable in $\omega$ for each fixed $E \in \mathcal{B}(\mathbb{R}_+ \times \mathbb{R}^d)$. Moreover, for each $0 \leq r < s$, $F \in \widetilde{\mathcal{F}}^*_r$, and $A \in \mathcal{B}(\mathbb{R}^d)$, let $W(\omega, u, \xi) = \bm{1}_{F \times (r, s] \times A}(\omega, u, \xi)$. Then, the process
	\begin{equation*}
		(W * \lambda^*)_t
		= \bm{1}_F (M^*_{s \land t}(A) - M^*_{r \land t}(A))
	\end{equation*}
	is continuous (recall \eqref{eq:mu_cts}) and adapted to $\widetilde{\mathbb{F}}^*$, thus predictable. By another application of Dynkin's $\pi$-$\lambda$ theorem, this suffices to show that $\lambda^*$ is a predictable random measure (with respect to $\widetilde{\mathbb{F}}^*$). Then, we define $\nu^*(ds, d\xi) \coloneqq \bm{1}_{\{\xi \neq 0\}} (1 \land |\xi|^2)^{-1} \lambda^*(ds, d\xi)$, which is again a predictable random measure. It follows that for every measurable function $f: \mathbb{R}^d \to \mathbb{R}$ satisfying $|f(\xi)| \leq C (1 \land |\xi|^2)$, $\forall\, \xi \in \mathbb{R}^d$, for some constant $C > 0$,
	\begin{equation}\label{eq:nu_M}
		(f * \nu^*)_t
		= \int_{\mathbb{R}^d} \frac{f(\xi)}{1 \land |\xi|^2} \,M^*_t(d\xi),\quad t \geq 0,
	\end{equation}
	and this process is continuous as $|(f * \nu^*)_t - (f * \nu^*)_s| \leq C |M^*_t(\mathbb{R}^d) - M^*_s(\mathbb{R}^d)|$.
	
	Define the process $Y^*(h) \coloneqq Y^* - \sum_{s \leq \cdot} (\Delta Y^*_s - h(\Delta Y^*_s))$, which has bounded jumps as $\Delta Y^*(h) = h(\Delta Y^*)$. Define the process $\widetilde{C}^* \coloneqq C^* + (hh^\top) * \nu^*$.\footnote{The process $\widetilde{C}^*$ is the candidate for the \emph{modified second characteristic} of $Y^*$ associated with $h$. See \cite{MR1943877}, Definition~II.2.16. Eventually, one can show $\widetilde{C}^*_{ij} = \langle (Y^*(h) - B^*)_i, (Y^*(h) - B^*)_j \rangle$, where $Y^*(h) - B^*$ is a locally square integrable martingale.} Let $\mu^{Y^*}(dt, d\xi) \coloneqq \sum_{s > 0} \bm{1}_{\{\Delta Y^*_s \neq 0\}} \delta_{(s, \Delta Y^*_s)}(dt, d\xi)$ denote the integer-valued random measure associated with the jumps of $Y^*$. We may also define their counterparts $Y(h)$, $\widetilde{C}$ and $\mu^Y$ on the original probability space $\Omega$. Then, using Lemma~\ref{lem:loc_mtg} and \eqref{eq:nu_M}, one can show the following processes are local martingales on $(\Omega^*, \mathcal{F}^*, \mathbb{F}^*, \mathbb{Q})$:
	\begin{enumerate}[label=(\roman*), topsep=0.3em, noitemsep]
		\item $Y^*(h) - B^*$,
		
		\item $(Y^*(h) - B^*) (Y^*(h) - B^*)^\top - \widetilde{C}^*$,
		
		\item $f * \mu^{Y^*} - f * \nu^*$, where $f: \mathbb{R}^d \to \mathbb{R}$ is measurable and satisfies $|f(\xi)| \leq C (1 \land |\xi|^2)$, $\forall\, \xi \in \mathbb{R}^d$, for some constant $C > 0$.
	\end{enumerate}
	Since all these processes are c\`adl\`ag, it is easy to see that they are local martingales with respect to the right-continuous regularized filtration $\widetilde{\mathbb{F}}^*$. We also note that $C^*_t - C^*_s$ takes values in $\mathbb{S}^d_+$ for all $0 \leq s \leq t$, $\mathbb{Q}$-a.s. Thus, according to \cite{MR1943877}, Theorem~II.2.21, $Y^*$ is a semimartingale with characteristics triplet $(B^*, C^*, \nu^*)$ (associated with $h$) on the filtered probability space $(\Omega^*, \mathcal{F}^*, \widetilde{\mathbb{F}}^*, \mathbb{Q})$.
	
	Next, we take a sequence of functions $(f_k)_{k \in \mathbb{N}^*}$ which is a class $C_1(\mathbb{R}^d)$ (recall Definition~\ref{def:conv_det} and Remark~\ref{rem:conv_det}). This countable collection $(f_k)$ is convergence determining, thus measure determining. Without loss of generality, we may include functions $(h_i h_j)_{i, j=1}^d$ to the sequence $(f_k)$ and keep the same notation, where $h_i$ is the $i$-th component of $h$. For each $i, j = 1, ..., d$, let $k(i, j)$ be the index such that $h_i h_j = f_{k(i, j)}$. Although $h_i h_j$ does not vanish around $0$, it is a continuous function satisfying $|h_i h_j| \leq C(1 \land \lvert \cdot \rvert^2)$ for some constant $C > 0$. Also, adding a finite number of functions does no harm to our following arguments. We define the process
	\begin{equation}\label{eq:G*}
		G^*_{k, t}
		\coloneqq (f_k * \nu^*)_t
		= \int_{\mathbb{R}^d} \frac{f_k(\xi)}{1 \land |\xi|^2} \,M^*_t(d\xi).
	\end{equation}
	We also define its counterpart $G_k$ on the original space $\Omega$. In particular, $G_k$ has the form of a Riemann integral: $G_{k, t} = \int_0^t \int_{\mathbb{R}^d} f_k(\xi) \,\kappa_s(d\xi) \,ds$. By the definition of $k(i, j)$, we have
	\begin{equation}\label{eq:C_G}
		\widetilde{C}^*_{ij} = C^*_{ij} + G^*_{k(i, j)},\quad
		\widetilde{C}_{ij} = C_{ij} + G_{k(i, j)}.
	\end{equation}
	We define the $\mathbb{R}^d$-valued predictable process $b^* = (b^*_i)$, the $\mathbb{R}^{d^2}$-valued predictable process $c^* = (c^*_{ij})$, and the real-valued predictable processes $g^*_k$, $k \in \mathbb{N}^*$, via
	\begin{equation*}
		\begin{split}
			b^*_{i, t}
			&\coloneqq \bm{1}_{\mathbb{R}} \biggl(\liminf_{n \to \infty} \frac{B^*_{i, t} - B^*_{i, (t - 1/n)^+}}{1/n}\biggr),\\
			c^*_{ij, t}
			&\coloneqq \bm{1}_{\mathbb{R}} \biggl(\liminf_{n \to \infty} \frac{C^*_{ij, t} - C^*_{ij, (t - 1/n)^+}}{1/n}\biggr),\\
			g^*_{k, t}
			&\coloneqq \bm{1}_{\mathbb{R}} \biggl(\liminf_{n \to \infty} \frac{G^*_{k, t} - G^*_{k, (t - 1/n)^+}}{1/n}\biggr).
		\end{split}
	\end{equation*}
	Since $B$, $C$, $(G_k)$ are absolutely continuous $\mathbb{P}$-a.s., we know that $B^*$, $C^*$, $(G^*_k)$ are absolutely continuous $\mathbb{Q}$-a.s. Consequently,
	\begin{equation}\label{eq:AC_Qas}
		\begin{split}
			\mathbb{Q}\biggl(\int_0^t (|b^*_s| + |c^*_s| + |g^*_{k, s}|) \,ds < \infty,\, B^*_t &= \int_0^t b^*_s \,ds,\, C^*_t = \int_0^t c^*_s \,ds,\\ G^*_{k, t} &= \int_0^t g^*_{k, s} \,ds,\, \forall\, k \in \mathbb{N}^*,\, t \geq 0\biggr) = 1.
		\end{split}
	\end{equation}
	
	Set $Z^* = \Phi(Z_0^*, Y^*)$. Note that the joint law of $(Y^*, Z^*, B^*, C^*, (G^*_k))$ under $\mathbb{Q}$ agrees with the joint law of $(Y, Z, B, C, (G_k))$ under $\mathbb{P}$. Thus, \eqref{eq:char} and \eqref{eq:AC_Qas} imply that for Lebesgue-a.e.\ $t \geq 0$, the joint law of $(Y^*_t, Z^*_t, b^*_t, c^*_t, (g^*_{k, t}))$ under $\mathbb{Q}$ agrees with the joint law of $(Y_t, Z_t, b_t, c_t, (\int_{\mathbb{R}^d} f_k(\xi) \,\kappa_t(d\xi)))$ under $\mathbb{P}$. It follows that
	\begin{equation}\label{eq:exp_id}
		\mathbb{E}^{\mathbb{Q}}\biggl[\int_0^t f(Y^*_s, Z^*_s, b^*_s, c^*_s, (g^*_{k, s})) \,ds\biggr]
		= \mathbb{E}\biggl[\int_0^t f\biggl(Y_s, Z_s, b_s, c_s, \biggl(\int_{\mathbb{R}^d} f_k(\xi) \,\kappa_s(d\xi)\biggr)\biggr) \,ds\biggr],
	\end{equation}
	for all $t > 0$ and measurable $f: \mathbb{R}^{\mathbb{N}} \to \mathbb{R}$ such that either side (and then both sides) of \eqref{eq:exp_id} is well-defined. In particular, \eqref{eq:thm_asm} and \eqref{eq:exp_id} yield
	\begin{equation}\label{eq:int_cond}
		\mathbb{E}^{\mathbb{Q}}\biggl[\int_0^t (|b^*_s| + |c^*_s| + |g^*_{k, s}|) \,ds\biggr] < \infty,\quad
		\forall\, k \in \mathbb{N}^*,\, t > 0.
	\end{equation}
	Recall that at the beginning of this proof, we have already established the existence of $\widehat{b}$, $\widehat{c}$ and $\widehat{\kappa}$ satisfying \eqref{eq:cond_exp}. For each $k \in \mathbb{N}^*$, define $\widehat{g}_k: \mathbb{R}_+ \times \mathcal{E} \to \mathbb{R}$ via
	\begin{equation}\label{eq:ghat}
		\widehat{g}_k(t, z)
		\coloneqq \int_{\mathbb{R}^d} f_k(\xi) \,\widehat{\kappa}(t, z, d\xi),\quad
		t \geq 0,\, z \in \mathcal{E}.
	\end{equation}
	Then, by \eqref{eq:cond_exp}, \eqref{eq:cond_exp_2}, \eqref{eq:exp_id} and using \cite{MR3098443}, Lemma~5.2 twice, we deduce that for Lebesgue-a.e.\ $t \geq 0$,
	\begin{equation}\label{eq:cond_exp_3}
		\begin{split}
			\widehat{b}(t, Z^*_t) &= \mathbb{E}^{\mathbb{Q}}[b^*_t \,|\, Z^*_t],\\
			\widehat{c}(t, Z^*_t) &= \mathbb{E}^{\mathbb{Q}}[c^*_t \,|\, Z^*_t],\\
			\widehat{g}_k(t, Z^*_t) &= \mathbb{E}^{\mathbb{Q}}[g^*_{k, t} \,|\, Z^*_t],\quad
			\forall\, k \in \mathbb{N}^*.
		\end{split}
	\end{equation}
	From \eqref{eq:int_cond} and Jensen's inequality, we get
	\begin{equation}\label{eq:int_cond_2}
		\mathbb{E}^{\mathbb{Q}}\biggl[\int_0^t (|\widehat{b}(s, Z^*_s)| + |\widehat{c}(s, Z^*_s)| + |\widehat{g}_k(s, Z^*_s)|) \,ds\biggr] < \infty,\quad
		\forall\, k \in \mathbb{N}^*,\, t > 0.
	\end{equation}
	
	\medskip
	\emph{Step 2: Extended partitions.}\quad For $m \in \mathbb{N}^*$, set $N(m) \coloneqq m^2$. Define the stopping times $T^m_0 \coloneqq 0$, $T^m_i \coloneqq (U^* + i-1) / m$, $i = 1, ..., N(m)$, and $T^m_{N(m)+1} \coloneqq \infty$. Set $\mathcal{G}^m_0 = \mathcal{H}^m_0 \coloneqq \mathcal{F}^*_0 = \sigma(U^*, Z_0^*)$, $\mathcal{G}^m_i \coloneqq \sigma(U^*, Z^*_{T^m_i})$, $i = 1, ..., N(m)$, and $\mathcal{H}^m_i \coloneqq \mathcal{G}^m_{i-1} \lor \sigma(\Delta(X^{T^m_i}, T^m_{i-1}))$, $i = 1, ..., N(m)+1$. Then, by Lemma~\ref{lem:ext_part}~(i), $\Pi^m \coloneqq (T^m_i, \mathcal{G}^m_i)_{i=0}^{N(m)}$ is an extended partition.
	
	\medskip
	\emph{Step 3: Concatenated probability measures.}\quad According to Theorem~\ref{thm:concat_meas}, for each $m \in \mathbb{N}^*$, there exists a unique probability measure $\mathbb{Q}^m \coloneqq \mathbb{Q}^{\otimes \Pi^m}$ on $(\Omega^*, \mathcal{F}^*)$ such that
	\begin{enumerate}[label=(\roman*), topsep=0.3em, noitemsep]
		\item $\mathbb{Q}^m(A) = \mathbb{Q}(A)$, for all $A \in \mathcal{H}^m_i$, $i = 0, ..., N(m)+1$,
		
		\item $\mathbb{Q}^m(B \,|\, \mathcal{F}^*_{T^m_i}) = \mathbb{Q}(B \,|\, \mathcal{G}^m_i)$, for all $B \in \mathcal{H}^m_{i+1}$, $i = 0, ..., N(m)$.
	\end{enumerate}
	From Step 1, we already know that the following processes are local martingales under $\mathbb{Q}$:
	\begin{enumerate}[label=(\roman*), topsep=0.3em, noitemsep]
		\item $Y^*(h) - B^*$,
		
		\item $(Y^*(h) - B^*) (Y^*(h) - B^*)^\top - \widetilde{C}^*$,
		
		\item $f_k * \mu^{Y^*} - f_k * \nu^*$, $k \in \mathbb{N}^*$.
	\end{enumerate}
	By Lemma~\ref{lem:mtg_1} and Lemma~\ref{lem:mtg_2}, the processes above are also local martingales under $\mathbb{Q}^m$. Here we may use either $\mathbb{F}^*$ or $\widetilde{\mathbb{F}}^*$ due to the right-continuity of sample paths. Also, by Proposition~\ref{prop:sample_path}~(ii), $C^*_t - C^*_s$ takes values in $\mathbb{S}^d_+$ for all $0 \leq s \leq t$, $\mathbb{Q}^m$-a.s. Thus, applying \cite{MR1943877}, Theorem~II.2.21, we deduce that $Y^*$ is a semimartingale with characteristics triplet $(B^*, C^*, \nu^*)$ (associated with $h$) on $(\Omega^*, \mathcal{F}^*, \widetilde{\mathbb{F}}^*)$ under each $\mathbb{Q}^m$.
	
	We also note that by Proposition~\ref{prop:sample_path}~(i) and the definition of $b^*$, $c^*$, $(g^*_k)$, for each $m \in \mathbb{N}^*$,
	\begin{equation}\label{eq:AC_Qmas}
		\begin{split}
			\mathbb{Q}^m\biggl(\int_0^t (|b^*_s| + |c^*_s| + |g^*_{k, s}|) \,ds < \infty,\, B^*_t &= \int_0^t b^*_s \,ds,\, C^*_t = \int_0^t c^*_s \,ds,\\ G^*_{k, t} &= \int_0^t g^*_{k, s} \,ds,\, \forall\, k \in \mathbb{N}^*,\, t \geq 0\biggr) = 1.
		\end{split}
	\end{equation}
	
	\medskip
	\emph{Step 4: Tightness and convergence.}\quad By \eqref{eq:C_G}, \eqref{eq:AC_Qas}, \eqref{eq:int_cond}, \eqref{eq:AC_Qmas}, and Corollary~\ref{cor:tight}, we know that each one of the following collections of probability measures is tight:
	\begin{enumerate}[label=(\roman*), topsep=0.3em, noitemsep]
		\item $(\mathbb{Q}^m \circ (B^*)^{-1})_{m \in \mathbb{N}^*}$ (on $C^d_0$),
		
		\item $(\mathbb{Q}^m \circ (\widetilde{C}^*)^{-1})_{m \in \mathbb{N}^*}$ (on $C^{d^2}_0$),
		
		\item $(\mathbb{Q}^m \circ (G^*_k)^{-1})_{m \in \mathbb{N}^*}$ (on $C^1_0$), for each $k \in \mathbb{N}^*$.
	\end{enumerate}
	Moreover, let $\varepsilon > 0$, $T > 0$, $a > 0$. Take $p(a) \in \mathbb{Q}$ such that $2/a < p(a) < 3/a$, so we have $p(a) \to 0$ as $a \to \infty$. Let $k(a) \in \mathbb{N}^*$ be the index such that $f_{k(a)} = (p(a) \lvert \cdot \rvert - 1)^+ \land 1$. It is easy to check that $f_{k(a)} \geq \bm{1}_{\{\lvert \cdot \rvert > a\}}$, and $f_{k(a)} \to 0$ pointwisely as $a \to \infty$. Then, applying Proposition~\ref{prop:exp_id} to $G^*_{k(a)}$ and $g^*_{k(a)}$, we get
	\begin{equation}\label{eq:tight_est}
		\begin{split}
			&\mathbb{E}^{\mathbb{Q}^m}[\nu^*([0, T] \times \{\xi: |\xi| > a\})]
			= \mathbb{E}^{\mathbb{Q}^m}\biggl[\int_{\{\xi: |\xi| > a\}} \frac{1}{1 \land |\xi|^2} \,M^*_T(d\xi)\biggr]\\
			&\quad\quad\quad\quad\leq \mathbb{E}^{\mathbb{Q}^m}\biggl[\int_{\mathbb{R}^d} \frac{f_{k(a)}(\xi)}{1 \land |\xi|^2} \,M^*_T(d\xi)\biggr]
			= \mathbb{E}^{\mathbb{Q}^m}\bigl[G^*_{k(a), T}\bigr]
			= \mathbb{E}^{\mathbb{Q}^m}\biggl[\int_0^T g^*_{k(a), s} \,ds\biggr]\\
			&\quad\quad\quad\quad= \mathbb{E}^{\mathbb{Q}}\biggl[\int_0^T g^*_{k(a), s} \,ds\biggr]
			= \mathbb{E}^{\mathbb{Q}}\bigl[G^*_{k(a), T}\bigr]
			= \mathbb{E}^{\mathbb{Q}}\biggl[\int_{\mathbb{R}^d} \frac{f_{k(a)}(\xi)}{1 \land |\xi|^2} \,M^*_T(d\xi)\biggr].
		\end{split}
	\end{equation}
	By \eqref{eq:thm_asm}, we have $\mathbb{E}^{\mathbb{Q}}[M^*_T(\mathbb{R}^d)] = \mathbb{E}[M_T(\mathbb{R}^d)] < \infty$. Thus, \eqref{eq:tight_est} and the dominated convergence theorem yield that
	\begin{equation}\label{eq:tight_est2}
		\lim_{a \to \infty} \sup_{m \in \mathbb{N}^*}\mathbb{Q}^m(\nu^*([0, T] \times \{\xi: |\xi| > a\}) > \varepsilon)
		\leq \lim_{a \to \infty} \frac{1}{\varepsilon} \mathbb{E}^{\mathbb{Q}}\biggl[\int_{\mathbb{R}^d} \frac{f_{k(a)}(\xi)}{1 \land |\xi|^2} \,M^*_T(d\xi)\biggr]
		= 0.
	\end{equation}
	Therefore, given \eqref{eq:tight_est2} and the tightness of (i-iii) above, \cite{MR1943877}, Theorem~VI.4.18 tells us that the collection of measures $(\mathbb{Q}^m \circ (Y^*)^{-1})_{m \in \mathbb{N}^*}$ on $D^d_0$ is tight. Since $Z_0^*$ has the same law under every $\mathbb{Q}^m$ (recall $Z_0^*$ is $\mathcal{H}^m_0 = \mathcal{F}^*_0$-measurable), the collection of measures $(\mathbb{Q}^m \circ (Z_0^*, Y^*)^{-1})_{m \in \mathbb{N}^*}$ on $\widehat{\Omega} \coloneqq \mathcal{E} \times D^d_0$ is tight. Passing to a convergent subsequence if necessary, with an abuse of notation, we may assume $\mathbb{Q}^m \circ (Z_0^*, Y^*)^{-1}$ converges weakly to a Borel probability measure $\widehat{\mathbb{P}}$ on $\widehat{\Omega}$, as $m \to \infty$.
	
	On the space $\widehat{\Omega} = \mathcal{E} \times D^d_0$, we denote the projections by $(\widehat{Z}_0, \widehat{Y})$. Let $\widehat{\mathcal{F}} \coloneqq \sigma(\widehat{Z}_0, \widehat{Y})$ be the Borel $\sigma$-algebra, and $\widehat{\mathcal{F}}_t \coloneqq \sigma(\widehat{Z}_0, \widehat{Y}^t)$ for $t \geq 0$. Set $\widehat{Z} \coloneqq \Phi(\widehat{Z}_0, \widehat{Y})$. The continuous mapping theorem then yields that the law of $(Y^*, Z^*)$ on $D^d_0 \times D^{\mathcal{E}}$ under $\mathbb{Q}^m$ converges weakly to the law of $(\widehat{Y}, \widehat{Z})$ on $D^d_0 \times D^{\mathcal{E}}$ under $\widehat{\mathbb{P}}$, i.e.\ $\mathbb{Q}^m \circ (Y^*, Z^*)^{-1} \Rightarrow \widehat{\mathbb{P}} \circ (\widehat{Y}, \widehat{Z})^{-1}$.
	
	\medskip
	\emph{Step 5: Agreement of one-dimensional laws.}\quad From Step 4, we know that the law of $Z^*$ on $D^{\mathcal{E}}$ under $\mathbb{Q}^m$ converges weakly to the law of $\widehat{Z}$ on $D^{\mathcal{E}}$ under $\widehat{\mathbb{P}}$. Since $\widehat{Z}$ is c\`adl\`ag, we know there exists a countable set $N \subset \mathbb{R}_+$ such that $\widehat{\mathbb{P}}(\widehat{Z}_t = \widehat{Z}_{t-}) = 1$ for every $t \notin N$. In other words, the projection map $D^{\mathcal{E}} \ni z \mapsto z(t) \in \mathcal{E}$ is continuous $(\widehat{\mathbb{P}} \circ \widehat{Z}^{-1})$-a.s.\ for every $t \notin N$. Thus, the continuous mapping theorem implies that the law of $Z^*_t$ on $\mathcal{E}$ under $\mathbb{Q}^m$ converges weakly to the law of $\widehat{Z}_t$ on $\mathcal{E}$ under $\widehat{\mathbb{P}}$, for every $t \notin N$.
	
	On the other hand, by Lemma~\ref{lem:ext_part}~(ii), the law of $Z^*_t$ under $\mathbb{Q}^m$ agrees with the law of $Z^*_t$ under $\mathbb{Q}$ for every $m \in \mathbb{N}^*$. This gives us $\widehat{\mathbb{P}} \circ (\widehat{Z}_t)^{-1} = \mathbb{Q} \circ (Z^*_t)^{-1} = \mathbb{P} \circ (Z_t)^{-1}$, for every $t \notin N$. Finally, by the right-continuity of the sample paths of $\widehat{Z}$ and $Z$, we have that $\widehat{Z}_s \to \widehat{Z}_t$ in law and $Z_s \to Z_t$ in law as $s \downarrow t$, for every $t \geq 0$. Therefore, we conclude that $\widehat{\mathbb{P}} \circ (\widehat{Z}_t)^{-1} = \mathbb{P} \circ (Z_t)^{-1}$ for every $t \geq 0$. This proves item (ii) of the theorem.
	
	\medskip
	\emph{Step 6: Characteristics of the limit.}\quad It remains to show that on the filtered probability space $(\widehat{\Omega}, \widehat{\mathcal{F}}, (\widehat{\mathcal{F}}_t)_{t \geq 0}, \widehat{\mathbb{P}})$ (strictly speaking, one needs to replace $(\widehat{\mathcal{F}}_t)_{t \geq 0}$ by its right-continuous regularization), $\widehat{Y}$ is a semimartingale with characteristics triplet $(\widehat{B}, \widehat{C}, \widehat{\nu})$ (defined in \eqref{eq:char_2}) associated with $h$. We define the following processes on $\widehat{\Omega}$ (recall \eqref{eq:ghat}):
	\begin{equation*}
		\widehat{G}_{k, t}
		\coloneqq (f_k * \widehat{\nu})_t
		= \int_0^t \int_{\mathbb{R}^d} f_k(\xi) \,\widehat{\kappa}(s, \widehat{Z}_s, d\xi) \,ds
		= \int_0^t \widehat{g}_k(s, \widehat{Z}_s) \,ds,\quad
		k \in \mathbb{N}^*.
	\end{equation*}
	We also define the process $\widehat{C}^\prime \coloneqq \widehat{C} + (hh^\top) * \widehat{\nu}$. Similar to \eqref{eq:C_G}, one has $\widehat{C}^\prime_{ij} = \widehat{C}_{ij} + \widehat{G}_{k(i,j)}$. If we manage to prove the following:
	\begin{enumerate}[label=(\roman*), topsep=0.3em, noitemsep]
		\item $\mathbb{Q}^m \circ (Y^*, B^*, \widetilde{C}^*)^{-1} \Rightarrow \widehat{\mathbb{P}} \circ (\widehat{Y}, \widehat{B}, \widehat{C}^\prime)^{-1}$,
		
		\item $\mathbb{Q}^m \circ (Y^*, G^*_k)^{-1} \Rightarrow \widehat{\mathbb{P}} \circ (\widehat{Y}, \widehat{G}_k)^{-1}$, for each $k \in \mathbb{N}^*$,
	\end{enumerate}
	then applying \cite{MR1943877}, Theorem~IX.2.4 (here we need $h$ to be continuous), we would finish the proof of item (i) of the theorem.
	
	On the other hand, on the space $\Omega^*$ we define the following processes:
	\begin{equation*}
		\overline{B}_t \coloneqq \int_0^t \widehat{b}(s, Z^*_s) \,ds,\quad
		\overline{C}_t \coloneqq \int_0^t \widehat{c}(s, Z^*_s) \,ds,\quad
		\overline{G}_{k, t} \coloneqq \int_0^t \widehat{g}_k(s, Z^*_s) \,ds,\quad
		k \in \mathbb{N}^*.
	\end{equation*}
	We also define the $\mathbb{R}^{d^2}$-valued process $\overline{C}^\prime$ via $\overline{C}^\prime_{ij} \coloneqq \overline{C}_{ij} + \overline{G}_{k(i, j)}$. Recall that in Step 4 we showed $\mathbb{Q}^m \circ (Y^*, Z^*)^{-1} \Rightarrow \widehat{\mathbb{P}} \circ (\widehat{Y}, \widehat{Z})^{-1}$, and in Step 5 we showed $\mathbb{Q}^m \circ (Z^*_t)^{-1} = \mathbb{Q} \circ (Z^*_t)^{-1}$ for all $m \in \mathbb{N}^*$ and $t \geq 0$. Using \eqref{eq:int_cond_2}, Proposition~\ref{prop:wconv_int} and Lemma~\ref{lem:joint_wconv}, we know these processes are well-defined $\mathbb{Q}$-a.s.\ and $\mathbb{Q}^m$-a.s.\ for each $m \in \mathbb{N^*}$, and we get the following weak convergence:
	\begin{enumerate}[label=(\roman*), topsep=0.3em, noitemsep]
		\item $\mathbb{Q}^m \circ (Y^*, \overline{B}, \overline{C}^\prime)^{-1} \Rightarrow \widehat{\mathbb{P}} \circ (\widehat{Y}, \widehat{B}, \widehat{C}^\prime)^{-1}$,
		
		\item $\mathbb{Q}^m \circ (Y^*, \overline{G}_k)^{-1} \Rightarrow \widehat{\mathbb{P}} \circ (\widehat{Y}, \widehat{G}_k)^{-1}$, for each $k \in \mathbb{N}^*$. 
	\end{enumerate}
	
	Thus, to conclude the proof, we need to show as $m \to \infty$: $\mathbb{Q}^m \circ (Y^*, B^*, \widetilde{C}^*)^{-1}$ and $\mathbb{Q}^m \circ (Y^*, \overline{B}, \overline{C}^\prime)^{-1}$ have the same limit; $\mathbb{Q}^m \circ (Y^*, G^*_k)^{-1}$ and $\mathbb{Q}^m \circ (Y^*, \overline{G}_k)^{-1}$ have the same limit. To do this, it suffices to show that for any $\varepsilon > 0$ and $t > 0$,
	\begin{equation}\label{eq:ucp}
		\begin{split}
			\lim_{m \to \infty} \mathbb{Q}^m\biggl(\max_{s \leq t} |B^*_s - \overline{B}_s| \geq \varepsilon\biggr) &= 0,\\
			\lim_{m \to \infty} \mathbb{Q}^m\biggl(\max_{s \leq t} |\widetilde{C}^*_s - \overline{C}^\prime_s| \geq \varepsilon\biggr) &= 0,\\
			\lim_{m \to \infty} \mathbb{Q}^m\biggl(\max_{s \leq t} |G^*_{k, s} - \overline{G}_{k, s}| \geq \varepsilon\biggr) &= 0,\quad
			\forall\, k \in \mathbb{N}^*.
		\end{split}
	\end{equation}
	Given \eqref{eq:int_cond} and \eqref{eq:cond_exp_3} (thus \eqref{eq:int_cond_2} as well), we know that \eqref{eq:ucp} is the consequence of Lemma~\ref{lem:approx}, and we are done.
\end{proof}

\begin{remark}\label{rem:canonical_space}
	We make a few comments on our choice of the canonical space in the proof of Theorem~\ref{thm:mp}, especially the component $C^{\mathcal{M}_+, d}_{0, \mathrm{i}}$. The canonical space for $(Y, B, C)$ is straightforward. The main difficulty is to find a proper space to fit in the third characteristic $\nu$, or some object from which we can recover $\nu$.
	
	Our first attempt is the space $C^{\mathbb{N}}_0 \coloneqq C^{\mathbb{R}^{\mathbb{N}}}_0$, on which the projections are denoted by $(G^*_k)$. We impose the probability measure $\mathbb{Q} = \mathrm{Leb}([0, 1]) \otimes (\mathbb{P} \circ (Z_0, Y, B, C, (G_k))^{-1})$ on our canonical space, where $G_k \coloneqq f_k * \nu$. Thus, we expect $G^*_k$ to be of the form $f_k * \nu^*$, where $\nu^*$ is the candidate of the third characteristic of $Y^*$ under $\mathbb{Q}$. However, it is not easy to construct $\nu^*$ explicitly from $(G^*_k)$. The best we can do is to show $Y^*$ is a semimartingale under $\mathbb{Q}$, so it has a third characteristic $\nu^*$. From this we can only show $G^*_k = f_k * \nu^*$ $\mathbb{Q}$-a.s. The measurability of $\Delta(f_k * \nu^*, T)$ with respect to $\sigma(\Delta(X, T))$ is not clear. Also, proving $G^*_k = f_k * \nu^*$ $\mathbb{Q}^m$-a.s.\ is not easy, since $\nu^*$ is constructed in a probability measure-specific way.
	
	Our second attempt is the space $\mathbb{M}$ consisting of all $\sigma$-finite positive measures $\nu$ on $\mathbb{R}_+ \times \mathbb{R}^d$ which admit a disintegration of the form $\nu(dt, dx) = \kappa(t, dx) dt$. We denote the projection to this space by $\nu^*$. We impose the probability measure $\mathbb{Q} = \mathrm{Leb}([0, 1]) \otimes (\mathbb{P} \circ (Z_0, Y, B, C, \nu)^{-1})$ on our canonical space. Then, we can show $\nu^*$, which is a component itself, is the third characteristic of $Y^*$ under $\mathbb{Q}$. However, one of the difficulties now becomes how to put a proper topology on $\mathbb{M}$ to make it Polish. Also, it is not obvious how to define the operators $\Theta$, $\nabla$ and $\Delta$ on $\mathbb{M}$. This makes the construction of the concatenated probability measures less straightforward.
	
	Therefore, we choose a space for the third characteristic which lies somewhere between the above two attempts. The space $C^{\mathcal{M}_+, d}_0$ is a function space, which is more tractable than $\mathbb{M}$. Meanwhile, elements of $C^{\mathcal{M}_+, d}_0$ are measure-valued functions, which encode richer information than $C^{\mathbb{N}}_0$. However, one problem with this space is that it is not $\Delta$-stable. By enlarging this space to $C^{\mathcal{M}, d}_0 \coloneqq C^{\mathcal{M}(\mathbb{R}^d)}_0$, one may solve the $\Delta$-stableness issue, but then Polishness fails. Instead, we restrict to the space $C^{\mathcal{M}_+, d}_{0, \mathrm{i}}$ of increasing trajectories, which is both $\Delta$-stable and Polish. The canonical process $M^*$ on this space is not the third characteristic itself, but a type of running integral. Using Lemma~\ref{lem:cstr_nu}, one can recover $\nu^*$ explicitly from $M^*$ in a probability measure-free way.
\end{remark}

\bibliography{bibliography}
\bibliographystyle{abbrv}

\end{document}